\documentclass[11pt]{article}
\usepackage[english]{babel}

\usepackage[T1]{fontenc}
\usepackage{lmodern}

\usepackage[utf8]{inputenc}
\usepackage[a4paper,margin=2.5cm]{geometry}

\usepackage{amssymb}
\usepackage{amsmath}
\usepackage{amsthm}
\usepackage{parskip}
\usepackage{longtable}
\usepackage{booktabs}
\usepackage{array}
\usepackage{setspace}
\theoremstyle{plain}
\newtheorem{theorem}{Theorem}
\numberwithin{theorem}{section}
\newtheorem*{namedtheorem}{Theorem}

\theoremstyle{plain}
\newtheorem{remark}{Remark}
\numberwithin{remark}{section}

\theoremstyle{plain}
\newtheorem{lemma}{Lemma}
\numberwithin{lemma}{section}

\theoremstyle{plain}
\newtheorem{proposition}{Proposition}
\numberwithin{proposition}{section}

\theoremstyle{plain}
\newtheorem{corollary}{Corollary}
\numberwithin{corollary}{section}

\numberwithin{equation}{section}

\title{Scaling limits in dependent random environments: relating a random walk, a branching process and a spatial branching process}
\author{Douglas Buchanan}
\date{December 2025}

\begin{document}

\maketitle

\begin{abstract}
    We extend existing connections between random walks, branching processes and spatial branching processes, and their respective scaling limits, to include processes in dependent random environments. More specifically, we prove new scaling limits of a random walk in a dependent random environment, an associated branching process in a dependent random environment, and a spatial branching process in a dependent random environment. We show that the scaling limits are related in ways reminiscent of existing results in fixed environments. A Ray-Knight Theorem relates the scaling limits of the random walk in the random environment and the branching process in the random environment. The Brownian snake relates the scaling limit of the spatial branching process in the random environment to the scaling limit of the random walk in the random environment.
\end{abstract}

\section{Introduction}
\subsection{Branching processes in random environments and a Ray-Knight theorem}
The existence of a critical branching process inside a positive excursion of the random walk was first pointed out in [12]. We recall that if $S$ is the nearest neighbour random walk on $\mathbb{Z}$, and we define
\begin{equation}
    M(i)=\sum_{j=0}^{T^{0}_1-1}1_{|S_j|=i,|S_{j+1}|=i+1},
\end{equation}
where $T^0_1:=\inf\{n\geq 1 : S_n=0\}$ is the first return time to 0, then $M$ is a critical branching process with offspring distribution Geom($\frac{1}{2}$).

This connection between random walks and branching processes is robust in the sense that it extends to random walks in random environments (RWRE) and branching processes in random environments (BPRE). This paper is concerned with exploiting this connection to prove new scaling limits of RWRE and BPRE, and to show the respective scaling limits are related. We shall also show that this same connection allows for a scaling limit of a spatial branching process in a random environment, and that the scaling limit of the spatial branching process is also related to the scaling limit of the RWRE via a Brownian snake which we will define. We first desribe the connection between the scaling limits in a fixed environment.

If we let 
\[
M^n(t)=\frac{M(\lfloor nt \rfloor)}{n},\; t\geq0, \; M^n(0):=1,
\]
where $M$ is defined in (1.1), then it is well-known that
\[
\{M^n(t)\; ; \; t\geq0\} \Rightarrow \{\eta(t)\; ; \; t\geq0\}
\]
in $D[0,\infty)$, the Skorohod space of cadlag functions from $[0,\infty)$ to $\mathbb{R}$, where the process $\eta$ is the unique solution to
\[
d\eta(t)=\sqrt{2\eta}dB(t), \; \eta(0)=1,
\]
and where $B$ is a standard Brownian motion.

Of course, $S$ also scales to Brownian motion, and the Ray-Knight Theorem then makes precise in what way the process $\eta$ is related to Brownian motion.
\begin{namedtheorem}[Classical Ray-Knight Theorem]
    \[
    \{\eta(t)\; ; \; t\geq0\} \stackrel{d}{=}\{l(x,\psi)\; ; \; x\geq 0\},
    \]
    where $l(x,s)$ is the local time of Brownian motion at level $x$ by time $s$, and $\psi$ is the first time the local time of Brownian motion at level 0 is greater than 1.
\end{namedtheorem}
We note that $\psi$ is defined by
\begin{equation}
\psi =\inf\{s:l(0,s)>1\}.
\end{equation}
It was noted in [12] that the existence of $M$ defined in (1.1) provides a form of intuition for this Ray-Knight Theorem. Indeed, we can interpret $M(i)$ in (1.1) as defining a discrete local time for the random walk $S$ at level $i$ by time $T^0_1$, where $T^0_1$ can then also be interpreted as the first time that the local time of $S$ at level 0 is greater than 1. The Ray-Knight theorem then states that this identity persists into the respective scaling limits of $M$ and $S$.

We describe how this connection between random walks and branching processes extends to the case of random environments.

Define $(\alpha_i)_{i\in\mathbb{Z}}$ to be a sequence of i.i.d. random variables such that
\begin{equation}
\begin{aligned}
    &1) \; \exists v\in (0,\frac{1}{2}) \text{ such that for all } i, \alpha_i\in (v,1-v), \\
    &2)\; \mathbb{E}\left[\ln\left(\frac{1-\alpha_i}{\alpha_i}\right)\right]=0, \\
    &3) \; \mathbb{E}\left[\left(\ln\left(\frac{1-\alpha_i}{\alpha_i}\right)\right)^2\right]=\sigma^2<\infty.
\end{aligned}
\end{equation}

The sequence $(\alpha_i)_{i\in \mathbb{Z}}$ can be considered to be a random environment for a random walk $\widetilde{S}$, where
\begin{equation}
    \mathbb{P}(\widetilde{S}_{j+1}=i+1|\widetilde{S}_j=i,\alpha_i)=\alpha_i.
\end{equation}
We define a rescaling of the environment $(\alpha^{(n)}_i)_{i \in \mathbb{Z}}$ such that
\begin{equation}
    \ln\left(\frac{1-\alpha^{(n)}_i}{\alpha^{(n)}_i}\right)\stackrel{d}{=}\frac{1}{\sqrt{n}}\ln\left(\frac{1-\alpha_i}{\alpha_i}\right),
\end{equation}
so that
\begin{equation}
    \left\{\sum_{i=1}^{\lfloor nx \rfloor}\ln\left(\frac{1-\alpha^{(n)}_i}{\alpha^{(n)}_i}\right)\; ; \; x \in \mathbb{R}\right\} \Rightarrow \{W(x)\; ; \; x \in \mathbb{R}\},
\end{equation}
in $D(-\infty,\infty)$, where $W$ is a two sided $\sigma$-Brownian motion.

Let $\widetilde{S}^n$ now be the nearest neighbour random walk in the random environment $(\alpha^{(n)}_i)_{i\in \mathbb{Z}}$, defined exactly as in (1.4). Then 
\begin{equation}
\widetilde{M}^n(i)=\sum_{j=0}^{\widetilde{T}^{n,0}_1-1}1_{|\widetilde{S}^n_j|=i,|\widetilde{S}^n_{j+1}|=i+1}
\end{equation}
for
\begin{equation}
    \widetilde{T}^{n,0}_1:=\inf\{j\geq1:\widetilde{S}^n_j=0\},
\end{equation}
defines a sequence of branching processes in a random environment, where at generation $i$, the offspring distribution of $\widetilde{M}^n$ is Geom($1-\alpha^{(n)}_i$). 

It was shown in [5] that $\widetilde{S}^n$ satisfies
\begin{equation}
    \left\{\frac{1}{n}\widetilde{S}^n_{\lfloor n^2t \rfloor}\; ; \; t\geq0 \right\} \Rightarrow \{Y(t)\; ; \; t\geq0\}
\end{equation}
in $D[0,\infty)$, where $Y$ is Brownian motion in the random potential $W$. $Y$ is known as the Brox diffusion. We can define the Brox diffusion via speed and scale as follows:
\begin{equation}
    \begin{aligned}
        Y(t):=A_W^{-1}(B(M_W^{-1}(t))), \quad \text{where } \\
        A_W(y):=\int_0^ye^{W(x)}dx, \quad \text{and } \\
        M_W(t):=\int_0^t\text{exp}(-2W(A_W^{-1}(B(s))))ds, \quad t \geq 0 \\
    \end{aligned}
\end{equation}
where $B$ is a standard one-dimensional Brownian motion independent of $W$. We shall return to this definition in (1.19).

It was shown in [2] that if $\widetilde{M}^n(0)=1$ for all $n\geq1$, where $\widetilde{M}^n$ is defined in (1.7), then there exists a process $\widetilde{\eta}$ such that
\begin{equation}
    \left\{\frac{\widetilde{M}^n(\lfloor nt \rfloor)}{n}\; ; \; t\geq 0\right\} \Rightarrow \{\widetilde{\eta}(t)\; ; \; t\geq 0\}
\end{equation}
in $D[0,\infty)$.

The Ray-Knight Theorem in this set up, which was proved in [2], is then\newline
\begin{namedtheorem}[Ray-Knight Theorem in a random environment]
    \[
    \{\widetilde{\eta}(t)\; ; \; t\geq 0\} \stackrel{d}{=}\{L_Y(x,M_W(\psi))\; ; \; x\geq 0\},
    \]
    where $L_Y(x,s)$ is the local time of $Y$, the Brox diffusion, at level $x$ by time $s$, $M_W$ is defined in (1.10) and $\psi$ is defined in (1.2).
\end{namedtheorem}

The first main result of this paper extends this Ray-Knight Theorem to a much wider class of processes, where the environment $(\alpha_i)_{i\in \mathbb{Z}}$ is allowed to exhibit very strong correlation.

\subsection{Our model and Theorem 1.1}
We define the model we will be considering in this paper. Let $(\beta_i)_{i\in \mathbb{Z}}$ be a sequence of correlated random variables such that
\begin{equation}
\begin{aligned}
    &1) \; \exists v\in (0,\frac{1}{2}) \text{ such that for all } i, \beta_i\in (v,1-v), \\
    &2)\; \mathbb{E}\left[\ln\left(\frac{1-\beta_i}{\beta_i}\right)\right]=0, \\
\end{aligned}
\end{equation}
and such that there exists a deterministic sequence $(D_n)_{n\geq1}$ such that
\begin{equation}
    \left\{\frac{1}{D_n}\sum_{i=1}^{\lfloor nx \rfloor}\ln\left(\frac{1-\beta_i}{\beta_i}\right)\; ; \; x \in \mathbb{R}\right\} \Rightarrow \{W(x)\; ; \; x \in \mathbb{R}\}
\end{equation}
in $D(-\infty,\infty)$, for some conservative process $W$. A conservative process is defined in (1.20). For now we simply note that this is a relatively mild assumption, as any L\'{e}vy process is conservative, and most Gaussian processes are conservative.

\begin{remark}
    The assumption (1.13) is the key assumption that allows for all three main results in this paper to go through. We shall place this assumption in context in Section 1.5.
\end{remark}

As in Section 1.1, define $(\beta^{(n)}_i)_{i\in \mathbb{Z}}$ such that
\begin{equation}
    \ln\left(\frac{1-\beta^{(n)}_i}{\beta^{(n)}_i}\right)=\frac{1}{D_n}\ln\left(\frac{1-\beta_i}{\beta_i}\right).
\end{equation}

We may now define $\widetilde{S}^n$ to be the random walk in the random environment $(\beta^{(n)}_i)_{i \in \mathbb{Z}}$, so that
\begin{equation}
    \mathbb{P}(\widetilde{S}^n_{j+1}=i+1|\widetilde{S}^n_j=i,\beta^{(n)}_i)=\beta^{(n)}_i.
\end{equation}
In view of the definition of the Brownian snake in Section 1.4, and with an abuse of notation, we let
\begin{equation}
    \widetilde{S}^n_t=\frac{1}{n}\widetilde{S}^n_{\lfloor n^2t \rfloor}, \; t\geq0.
\end{equation}
We then let $\widetilde{M}^n$ be the rescaled branching process under the first $n$ excursions of $\widetilde{S}^n$, so that
\begin{equation}
    \widetilde{M}^n(i/n)=\frac{1}{n}\sum_{j=0}^{\widetilde{T}^{n,0}_n-1}1_{|S^n_{j/n^2}|=i/n,|S^n_{(j+1)/n^2}|=(i+1)/n},
\end{equation}
where $\widetilde{T}^{n,0}_n$ is the $n$th return time to 0 of $\widetilde{S}^n$. We note that $\widetilde{M}^n$ is then a branching process with initial mass 1, and branching times separated by $\frac{1}{n}$. We extend $\widetilde{M}^n$ to be defined for all $t$ by 
\begin{equation}
    \widetilde{M}^n(t)=\widetilde{M}^n\left(\frac{\lfloor nt \rfloor}{n}\right), \; t\geq0.
\end{equation}
In order to state our first main result we need to define Brownian motion in a random potential $Z$.

Given any real-valued process $Z=\{Z(x)\; ; \; x\in \mathbb{R}\}$ (which we take to be a random potential across space), we can define a continuous real-valued process $\widetilde{Z}$ moving through the potential $Z$, the $Z$-associated process, as follows:
\begin{equation}
    \begin{aligned}
        \widetilde{Z}(t):=A_Z^{-1}(B(M_Z^{-1}(t))), \quad \text{where } \\
        A_Z(y):=\int_0^ye^{Z(x)}dx, \quad \text{and } \\
        M_Z(t):=\int_0^t\text{exp}(-2Z(A_Z^{-1}(B(s))))ds, \quad t \geq 0 \\
    \end{aligned}
\end{equation}
where $B$ is a standard one-dimensional Brownian motion. We call the function $A_Z$ the scale function for the potential $Z$, or the scale function defined by the potential $Z$. Similarly, we call $dM_Z$ the speed measure for the potential $Z$, or the speed measure defined by the potential $Z$.\footnote{The form of the scale function and speed measure for the potential $Z$ follows from considering a formal solution to the SDE $d\widetilde{Z}_t=dB_t -\frac{1}{2}Z'(\widetilde{Z})dt$.} The process $\widetilde{Z}$ is then the process associated with the potential $Z$, or the $Z$-associated process.

We define a conservative potential to be a process $Z$ such that 
\begin{equation}
    B(M_Z^{-1}(t)) \in \text{ran}(A_Z), \quad \forall t\geq0 \text{ a.s.},
\end{equation}
so that the $Z$-associated process, $\widetilde{Z}$, exists as a random element of $C[0,\infty)$.

The following is the first main result of this paper, and is a version of the classical Ray-Knight theorem for our model.
\begin{theorem}
    The following two convergences hold, for $\widetilde{S}^n$ defined in (1.16) and $\widetilde{M}^n$ defined in (1.18):
    \[
    \left\{\widetilde{S}^n_t\; ; \; t\geq0 \right\} \Rightarrow \{Y(t)~\; ; \; t\geq 0\},
    \]
    and
    \[
    \left\{\widetilde{M}^n(t)\; ; \; t\geq0 \right\} \Rightarrow \{H(t)\; ; \; t\geq0\},
    \]
    in $D[0,\infty)$, where $Y$ is the $W$-associated process for $W$ defined in (1.13), and where
    \[
    \{H(t)\; ; \; t\geq0\}\stackrel{d}{=}\{L_Y(x,M_W(\psi))\; ; \; x\geq 0\},
    \]
    where $L_Y(x,t)$ is the local time of $Y$ at level $x$ by time $t$, $dM_W$ is the speed measure for the potential $W$, and $\psi$ is as defined in (1.2).
\end{theorem}

\subsection{Spatial branching processes in a random environment and Theorem 1.2}
The second two main results of this paper concern the scaling limit of the corresponding spatial branching process in our correlated random environment. Theorem 1.2 will show the existence of such a scaling limit, and Theorem 1.3 will make precise the relation between this scaling limit and the scaling limits in Theorem~1.1. We briefly describe a scaling limit of a spatial branching process in a fixed environment.

At time $t=0$ there are $n$ particles located at the origin in $\mathbb{R}^d$. These particles undergo independent Brownian motions until time $t=\frac{1}{n}$, at which point each particle gives birth to a random number of offspring (possibly 0) according to a given offspring distribution. Each newly born particle undergoes an independent Brownian motion, started from the location of their ancestor's death, for $t \in [\frac{1}{n}, \frac{2}{n})$. At time $t=\frac{2}{n}$ each particle again produces a random number of offspring according to the same offspring distribution. This pattern of splitting and spatial spreading continues over time.

We may define a measure-valued process by
\begin{equation}
X_t^n(1_A) = \frac{\#\{\text{particles in } A \text{ at time } t\}}{n}, \quad  A\subset \mathbb{R}^d,
\end{equation}
where $\mu(f)$ denotes the integral of a function $f$ with respect to a measure $\mu$. We call the process $X^n$ a branching Brownian motion

Then, under suitable conditions on the offspring distribution, (mean 1 and finite variance suffice),
\[
X^n \Rightarrow X
\]
in a suitable Skorohod space, where $X$ is known as a superprocess.

Spatial branching processes in random environments and their scaling limits are defined analagously, but where the offspring distribution is allowed to be random. We define the spatial branching process of interest for our model.

Fix $K>0$. At time $t=0$ there are $n$ particles located at the origin in $\mathbb{R}^d$. A particle alive at time $t=\frac{i}{n}\in[0,K)$ produces a random number of offspring according to a geometric distribution with parameter $1-\beta_i^{(n)}$, for $\beta^{(n)}_i$ defined in (1.14). In other words, if $N^n_i$ denotes the number of offspring of a particle reproducing at time $t=\frac{i}{n}\in[0,K)$, then
\begin{equation}
\mathbb{P}(N^n_i = k |  \beta_i^{(n)}) = (\beta_i^{(n)})^k(1-\beta_i^{(n)}) .  
\end{equation}
We require $X^n_K:=0$, and take the spatial motion to be independent of the environment $(\beta^{(n)}_i)_{i\geq1}$. Formally, the measure-valued process $X^n$ for our model is defined by
\begin{equation}
X_t^n(1_A) = \frac{\#\{\text{particles in } A \text{ at time } t\}}{n}, \quad t\in[0,K), \; A\subset \mathbb{R}^d,
\end{equation}
with $X^n_K:=0$, and branching mechanism given by (1.22). We call this process branching Brownian motion in a random environment (BBMRE). The reason for truncating the process $X^n$ at time $K$ is motivated by the Brownian snake representation of the process $X^n$, which is given in Section 1.4. 
\begin{remark}
    The total mass process for our model, $X^n(1_{\mathbb{R}^d})$, is the branching process $\widetilde{M}^n$ defined in (1.18), truncated at time $K$. This provides a connection between the process $X^n$ and the random walk $\widetilde{S}^n$ defined in (1.16). We shall exploit this connection when defining the Brownian snake in Section 1.4.
\end{remark}

We make the following extra assumptions:
\begin{equation}
    \begin{aligned}
        &1) \text{ The sequence $(\beta_i)_{i\geq 1}$ is positively correlated.} \\
        &2)\;  D_n\sim Var\left(\sum_{i=1}^n\ln\left(\frac{1-\beta_i}{\beta_i}\right)\right), \text{ for $D_n$ defined in (1.13).} \\
        &3) \; \exists j \in \mathbb{N}:\frac{1}{D_n^k}\mathbb{E}\left[\left|\sum_{i=1}^{n} \ln\!\left(\frac{1-\beta_i}{\beta_i}\right)\right|^k\right] \to 0 \text{ as } n \to \infty \text{ for all $k \geq j$}. \\
        &4) \; \text{The process $W$ defined in (1.13) is continuous},
    \end{aligned}
\end{equation} 
where $\sim$ denotes asymptotic equivalence. 

\begin{remark}
    Without loss of generality, and in the pursuit of notational elegance, we shall take $j=3$ in assumption 3) in (1.24) for the rest of this paper.
\end{remark}

We require some notation. Let $\mathcal{M}(\mathbb{R}^d)$ denote the space of finite, non-negative measures on $\mathbb{R}^d$, equipped with the weak topology. For a locally compact Polish space $E$, we let $D_E[0,\infty)$ (resp. $C_E[0,\infty)$) denote the space of c{\`a}dl{\`a}g (resp. continuous) functions from $[0,\infty)$ to $E$ endowed with the Skorohod (resp. uniform convergence on compacts) topology. We let $\mathcal{C}_b^2(E)$ denote the space of twice continuously differentiable and bounded real valued functions on $E$, and $\mathcal{B}(E)$ denote the set of bounded and measurable functions on $E$.

The following is the second main result of this paper.
\begin{theorem}
Under conditions (1.24), the sequence $\{X^n\}_{n\geq1}$ defined in (1.23) is tight in $D_{\mathcal{M}(\mathbb{R}^d)}[0,K]$. Any limit point $X$ is continuous on $[0,K)$, vanishes at $K$, and satisfies the following martingale problem:
\[
\forall \phi \in \mathcal{C}_b^2(\mathbb{R}^d),\; \forall t\in[0,K)\: : \:
X_t(\phi) - X_0(\phi) - \frac{1}{2}\int_0^tX_s(\Delta\phi)ds - \gamma_t(\phi)
\]
is a continuous martingale, with quadratic variation process
\[
2\int_0^tX_s(\phi^2)ds,
\]
where $\{\gamma_t(\phi);t \geq 0\}$ is the weak limit of the sequence
\[
\left \{\sum_{i=1}^{{\lfloor{nt}\rfloor}}X^n_{i/n}(\phi)\left (\frac{\beta_i^{(n)}}{1-\beta_i^{(n)}}-1\right)\; ; \; t \geq 0 \right\}_{n\geq1}.
\]
\end{theorem}
\begin{remark}
    A characterisation of the limit $X$ will follow from Theorem 1.3, presented at the end of Section 1.4.
\end{remark}

\begin{remark}
    In certain cases, the process $\{\gamma_t(\phi);t\geq0\}$ will be a stochastic integral process. We return to this point at the end of Section 4.
\end{remark}
\subsection{The Brownian snake and Theorem 1.3}
The existence of the branching process $\widetilde{M}^n$ inside the excursions of $\widetilde{S}^n$ noted in (1.17) allows us to recover the process $X^n$ defined in (1.23) from $\widetilde{S}^n$. It is the Brownian snake defined below that allows for this construction. Theorem~1.3 will show this connection between $X^n$ and $\widetilde{S}^n$ persists into their respective scaling limits.

We refer to [6] for a thorough introduction to the Brownian snake, to Chapter 3 in [1] for a more accessible introduction, and to Section 1.2 in [8] for an introduction to the method of discrete approximations to the continuous time Brownian snake. It is this last method which we employ here. The notation we define below draws heavily on that in Section 1.2 of [8].

Recall that $\widetilde{S}^n$ (defined in (1.16)) is the Donsker-rescaled nearest neighbour random walk in the random environment $(\beta^{(n)}_i)_{i \in \mathbb{Z}}$.

Define the function $f:D[0,\infty) \to D[0,\infty)$ by 
\begin{equation}
f(x)(t)=\begin{cases}
    |x(t)| & \text{if } -K\leq x(t) \leq K \\
    K-(|x(t)|-K) & \text{else}
\end{cases}
\end{equation}
where $K$ is the constant in the definition (1.23). 

We then define 
\begin{equation}
    {S}_t^n=f(\widetilde{S}^n)(t),
\end{equation} 
the Donsker-rescaled random walk in random environment $(\beta_i^{(n)})_{i \in \mathbb{Z}}$, reflected at 0 and $K$.

Let $\mathcal{W} \subset C_{\mathbb{R}^d}[0,\infty) \times \mathbb{R}_+$ denote the set of all stopped paths. By a stopped path we mean a pair $w=(\mathbf{w},\zeta) \in C_{\mathbb{R}^d}[0,\infty) \times \mathbb{R}_+$ such that $\mathbf{w}(t)=\mathbf{w}(\zeta)$ for all $t \geq \zeta$. The point $\zeta$ is often called the lifetime of the stopped path $w$ and will be denoted $\zeta_w$. The terminal point of $w$ is defined to be $w(\zeta_w)$, and is denoted $\widehat{w}$. 

We equip $\mathcal{W}$ with the distance 
\[
d(w,w')=\sup_{t \geq 0}|\mathbf{w}(t)-\mathbf{w}'(t)| + |\zeta_w-\zeta_{w'}|
\]
It follows that $(\mathcal{W},d)$ is a Polish space (see [6] for a more general statement).

We shall use the notion of a stopped path to define the spatial motion undergone by a single lineage of the process $X^n$ defined in (1.23), up until the time given by the lifetime of the stopped path.

The next definition defines the Brownian snake for our model. Let $B_1,B_2,...$ be a collection of independent Brownian motions in $\mathbb{R}^d$, stopped at time $t=\frac{1}{n}$, independent of ${S}^n$ and starting at the origin in $\mathbb{R}^d$. Define $\mathbf{W}^n_0 = 0\in\mathbb{R}^d$ and $(\mathbf{W}^n_{jn^{-2}})_{j\geq0}$ to be the path-valued process such that
\begin{equation}
\mathbf{W}^n_{(j+1)/{n^2}}(t)=
\begin{cases}
\mathbf{W}^n_{{j}/{n^2}}\left(t \wedge ({S}^n_{{j}/{n^2}} - \frac{1}{n})\right), & \text{if } {S}^n_{(j+1)/n^2} - {S}^n_{j/n^2}=-\frac{1}{n} \\
\mathbf{W}^n_{{j}/{n^2}} \odot B_j(t), & \text{if } {S}^n_{(j+1)/n^2} - {S}^n_{j/n^2}=\frac{1}{n}
\end{cases}
\end{equation}
where $B_j\odot B_k$ denotes the concatenation of two paths, defined in the obvious way. 

We then define $\mathbf{W}^n_{s}:=\mathbf{W}^n_{{\lfloor sn^2 \rfloor }/{n^2}}$, for $s\geq 0$.

We can now define the discrete Brownian snake to be the random element of $D_{\mathcal{W}}[0,\infty)$, with path as defined in (1.27) and lifetime process ${S}^n$. That is,
\[
\mathbb{W}^n_t:=(\mathbf{W}^n_t,{S}^n_t).
\]
We provide some intuition behind this construction. We start with the trivial path $\mathbf{W}^n_0=0\in\mathbb{R}^d$. Suppose we have defined the path $\mathbf{W}^n_{jn^{-2}}$. If $({S}^n_{(j+1)n^{-2}} - {S}^n_{jn^{-2}})=\frac{1}{n}$, to obtain the path $\mathbf{W}^n_{(j+1)n^{-2}}$ we extend the path $\mathbf{W}^n_{jn^{-2}}$ by a Brownian motion $B_j$, started from the terminal point of $\mathbf{W}^n_{jn^{-2}}$, which is given by $\widehat{\mathbb{W}}^n_{jn^{-2}}$. If instead ${S}^n_{(j+1)n^{-2}} - {S}^n_{jn^{-2}}=-\frac{1}{n}$, to obtain $\mathbf{W}^n_{(j+1)n^{-2}}$, we erase part of the path $\mathbf{W}^n_{jn^{-2}}$ from the terminal point $\widehat{\mathbb{W}}^n_{jn^{-2}}$ over a time interval of length $\frac{1}{n}$. Having defined the path $\mathbf{W}^n_{t}$ for each $t=\frac{j}{n^2}$, we then extend $\mathbf{W}^n_{\cdot}$ to be defined for all $t\in[0,\infty)$ by letting the path $\mathbf{W}^n_t$ be unchanged for $t\in [\frac{j}{n^2},\frac{j+1}{n^2})$. By construction, $\mathbb{W}^n_t=(\mathbf{W}^n_t,{S}^n_t)$ is then the random element of $D_{\mathcal{W}}[0,\infty)$, with path as defined in (1.16) and lifetime process ${S}^n$. Note that if we fix a path $\mathbf{W}^n_t$, then $\mathbf{W}^n_t:[0,{S}^n_t]\mapsto \mathbb{R}^d$.

The connection between the random walk $\widetilde{S}^n$ defined in (1.16) and the branching process $\widetilde{M}^n$ defined in (1.18) motivates the definition of the discrete Brownian snake we have given. In the spirit of (1.16) and (1.18), we can consider an excursion of ${S}^n$ to be a contour process coding a random tree. This random tree is the genealogy of one of the initial individuals in the branching process $\widetilde{M}^n$, up until time $K$. It then follows that $(\mathbb{\widehat{W}}^n_t)_{t\geq0}$ traces out the spatial motion of a branching Brownian motion with genealogy given by this random tree, one lineage at a time. It is not hard to believe, and indeed we will see below (see (1.31)), that when this procedure is carried out across $n$ excursions of $S^n$ that the resulting branching Brownian motion is in fact the process $X^n$ defined in (1.23).

We describe formally how the process $X^n$ can be recovered from the Brownian snake. Define the number of upsteps made by ${S}^n$ from level $\frac{i}{n}$ to level $\frac{i+1}{n}$ by time t as
\begin{equation}
    L^{n,{i}/{n}}_t=\sum_{j=0}^{\lfloor n^2t \rfloor}1_{{S}^n_{j/n^2}=i/n,{S}^n_{(j+1)/n^2}=(i+1)/n}.
\end{equation}
Recall ${S}^n$ has undergone a Donsker-rescaling, see (1.26).

Define the (discrete) local time of ${S}^n$ at level $s$ by time $t$ to be the rescaled number of upsteps from level $\frac{\lfloor ns \rfloor}{n}$ by time t. Formally,
\begin{equation}
L^{n,s}_t=\frac{1}{n}L^{n,{\lfloor ns \rfloor}/{n}}_t.
\end{equation}

Introduce the inverse local time of $S^n$ at level $a$ by time $t$ as
\begin{equation}
    T^{n,a}_t=\text{inf}\{s:L^{n,a}_s>t\}.
\end{equation}
$T^{n,a}_t$ is the first time that the number of upsteps made by $S^n$ from level $a$ exceeds $nt$. 

Then, for $\phi \in \mathcal{B}(\mathbb{R}^d)$, we have that
\begin{equation}
    X^n_{{i}/{n}}(\phi)=\int_{0}^{T^{n,0}_1}\phi(\widehat{\mathbb{W}}_s)L^{n,{i}/{n}}(ds)
\end{equation}
defines the same measure-valued process as in (1.23), where $L^{n,a}(ds)$ denotes the random measure with random distribution function $s \mapsto L^{n,a}_s$. This convention of using the same notation to denote an increasing function and the corresponding measure will be used throughout this paper.

\begin{remark}
    The definition of the process $X^n$ in (1.31) explains the truncation in the definition of the process $X^n$ in (1.23). Indeed, according to (1.31), 
    \[
    X^n_K(1_{\mathbb{R}^d})=L^{n,K}([0,T^{n,0}_1]).
    \]
    However, according to the definition of the local time $L^{n,a}$ in (1.29) and the definition of the contour process $S^n$ in (1.26), $L^{n,K}$ is the 0 measure, as the reflected random walk $S^n$ does not make any upsteps from level $K$. Thus $X^n_K=0$, which is consistent with the definition of $X^n$ in (1.23).
\end{remark}

We can now state the third main result of this paper.

\begin{theorem}
    Under (1.24), the pair $(\mathbf{W}^n,S^n)$ is C-tight in $D_{\mathcal{W}}[0,\infty)$ and converges weakly to a limiting $(\mathbf{W},S)$, where $S=f(Y)$ is the $W$-associated process reflected at 0 and $K$. Moreover, if we let $L^a_t$ be the local time of $S$ at level $a$ by time $t$, $T^a_t$ be the inverse local time of $S$ at level a by time t, and using Theorem~1.1 let $X$ be a limit point of the sequence $\{X^n\}_{n\geq1}$, then $\forall\phi\in\mathcal{B}(\mathbb{R}^d), \forall t \in [0,K]$,
    \[
    X_t(\phi)=\int_{0}^{T^{0}_1}\phi(\widehat{\mathbb{W}}_s)L^{t}(ds).
    \]
    In particular, any limit point $X$ of the sequence $\{X^n\}_{n\geq1}$ is uniquely characterised.
\end{theorem}
\subsection{Context}
Before proceeding with the proofs, we briefly place our work in context. The context for Theorem~1.1 was outlined in Section 1.1. 

Theorem~1.2 is similar to a result shown by Mytnik in [7], and Theorem~1.3 is similar to a result shown by Mytnik, Xiong and Zeitouni in [8]. We describe the model studied in [7]. 

We are given an i.i.d. mean 0 sequence of random fields $(\xi_i(x))_{i \geq 0}$ in $\mathbb{R}^d$, with covariance function 
\[
g(x,y)=\mathbb{E}[\xi_i(x)\xi_i(y)], \quad x,y \in \mathbb{R}^d.
\]
Let $X^n$ be a BBMRE where a particle located at position $x$ at time $t=\frac{i}{n}$ either dies with probability $\frac{1}{2} - \frac{\xi_i(x)}{2\sqrt{n}}$, or splits into two with probability $\frac{1}{2} + \frac{\xi_i(x)}{2\sqrt{n}}$.

It was shown in [7], under some additional technical assumptions on $(\xi_i(x))_{i \geq 0}$, that the sequence $\{X^n\}_{n\geq1}$ converges weakly in $D_{\mathcal{M}(\mathbb{R}^d)}[0,\infty)$ to some continuous superprocess in a random environment, $X$. The limit was shown to be the unique solution to the following martingale problem:
\begin{equation}
\forall \phi \in \mathcal{C}_b^2(\mathbb{R}^d): \quad 
X_t(\phi) - X_0(\phi) - \frac{1}{2}\int_0^tX_s(\Delta\phi)ds
\end{equation}
is a continuous martingale, with quadratic variation process
\begin{equation}
\int_0^tX_s(\phi^2)ds + \int_0^t\int_{\mathbb{R}^d}\int_{\mathbb{R}^d}g(x,y)\phi(x)\phi(y)X_s(dx)X_s(dy).
\end{equation}
This model was further studied and generalised by Mytnik, Xiong and Zeitouni in [8]. The $(\xi_i(x))_{i \geq 0}=(\xi_i^n(x))_{i \geq 0}$ were taken to have mean $\frac{\nu}{\sqrt{n}}$ and covariance function $g$ as above. At time $t=\frac{i}{n}$ a particle located at $x$ gives birth to a random number of offspring according to a geometric offspring distribution with parameter $\frac{1}{2} - \frac{\xi_i(x)}{4\sqrt{n}}$. 

The authors show that the measure-valued processes converges weakly in $D_{\mathcal{M}(\mathbb{R}^d)}[0,\infty)$ to a continuous superprocess in a random environment, $X$. They state that any limit point satisfies the following martingale problem:
\begin{equation}
\forall \phi \in \mathcal{C}_b^2(\mathbb{R}^d)\: : \:
X_t(\phi) - X_0(\phi) - \frac{1}{2}\int_0^tX_s(\Delta\phi)ds - \int_0^tX_s\left((\nu+\frac{1}{2}\bar{g})\phi\right)ds
\end{equation}
is a continuous martingale, with quadratic variation process
\[
2\int_0^tX_s(\phi^2)ds + \int_0^t\int_{\mathbb{R}^d}\int_{\mathbb{R}^d}g(x,y)\phi(x)\phi(y)X_s(dx)X_s(dy),
\]
where $\bar{g}=g(x,x)$.

The authors in [8] also characterise the limiting $X$ using a Brownian snake exactly as we do in Theorem~1.3. We shall return to this model at the end of section 4 to draw some comparisons between the martingale problem (1.34) and the generalisation we shall prove in Theorem~1.2.

Theorem~1.2 and Theorem~1.3  can be seen as providing an analogous treatment of BBMRE to that in [7] and [8], but where the correlation in the environments is taken to be across time, and not space.

We also briefly explain the importance of the assumption (1.13). Let $\Phi^n_{\xi_i(x)}$ denote the probability generating function of the geometric offspring distribution at generation $i$, conditioned on $\xi_i(x)$, the environment at location $x\in \mathbb{R}^d$ from the model in [8]. Then
\begin{equation}
(\Phi^n_{\xi_i(x)})'(1)=1+\frac{\xi_i(x)}{\sqrt{n}}+\frac{\xi_i(x)^2}{2n}+o\left(\frac{1}{n}\right),
\end{equation}
holds. 

If we define
\begin{equation}
    B^n_t(x):=\frac{1}{\sqrt{n}}\sum_{i=1}^{\lfloor nt \rfloor}\xi_i(x),
\end{equation}
then 
\begin{equation}
    B^n_t(x)\to B,
\end{equation}
also holds, where $\frac{\partial B_t(x)}{\partial t}$ is a Gaussian noise, white in time and coloured in space. We can view the $B^n$ as defining discrete potential functions for the branching processes, which by (1.37) converge to a non-trivial scaling limit. 

If we now let $\Phi_{\beta^{(n)}_i}$ denote the probability generating function of the Geom($1-\beta^{(n)}_i$) offspring distribution, conditioned on the environment $\beta^{(n)}_i$, then
\begin{equation}
    (\Phi_{\beta^{(n)}_i})'(1)=1-\frac{1}{D_n}\ln\left(\frac{1-\beta_i}{\beta_i}\right)+\frac{1}{2D_n^2}\left(\ln\left(\frac{1-\beta_i}{\beta_i}\right)\right)^2+o\left(\frac{1}{D_n^2}\right),
\end{equation}
can be seen to be true. Assumption (1.13) requires the analogue of (1.37) to be true for our model. It is interesting that (1.13) is also precisely what is needed to obtain a scaling limit for the random walk $\widetilde{S}^n$, defined in (1.15).

\subsection{Structure of the paper}
We shall prove Theorem~1.1 in Section 2. 

In Section 3 we will show convergence of the sequence of Brownian snakes and prove Theorem~1.3. The proof requires some preliminary work and its structure is outlined at the start of the section.

In Section 4 we prove Theorem~1.2. We will give an explicit semimartingale decomposition of $X^n_t(\phi)$, for a test function $\phi \in \mathcal{C}^2_b(\mathbb{R}^d)$. Tightness of the sequence $\{X^n\}_{n\geq1}$ will follow from tightness of all the terms in the semimartingale decomposition, and we derive the form of the martingale problem stated in Theorem~1.2 at the end of this Section. 

In Section 5 we give an explicit example of a sequence $(\beta_i)_{i \geq 1}$ satisfying our assumptions (1.12) for Theorem~1.1 and (1.24) for Theorems 1.2 and 1.3 that displays long range dependence.

In the Glossary section at the end of the paper we list frequently used notation.
\section{Proof of Theorem 1.1}
We shall first show weak convergence of the sequence $\{\widetilde{S}^n\}_{n\geq1}$ defined in (1.16). The method of proof is that used in [13]. In [13], $(\beta_i)_{i\geq1}$, the environment for the process $\widetilde{S}^1$, was taken to be i.i.d., and convergence to the Brox diffusion was shown, but we can transfer all of the results over to our situation. Since all the relevant proofs in [13] carry over, we defer to [13] for the details and only sketch the proofs of the main results.

First define the discrete potential
\begin{equation}
    {W}^n(x):=\sum_{i=1}^{\lfloor nx \rfloor} \ln\left(\frac{1-\beta^{(n)}_i}{\beta^{(n)}_i}\right)=\frac{1}{D_n}\sum_{i=1}^{\lfloor nx \rfloor} \ln\left(\frac{1-\beta_i}{\beta_i}\right), \; x \in \mathbb{R}.
\end{equation}
Note by (1.13) that $\{{W}^n(x) \; ; \; x\in \mathbb{R}\}$ converges weakly in $D(-\infty,\infty)$ to $\{W(x)\; ; \; x\in \mathbb{R}\}$, where $W$ is defined in (1.13). $W^n$ is known as the potential function for the random walk $\widetilde{S}^n$.

Then, recalling the definition of Brownian motion in a random potential in (1.19), define $Y^n$ to be the $W^n$-associated process and $Y$ to be the $W$-associated process. We denote by $A_{{W}^n}$ and $dM_{{W}^n}$ the scale function, respectively speed measure, defined by the potential ${W}^n$, and by $A_W$ and $dM_W$ the scale function, respectively speed measure, defined by the potential $W$.

There is a subtlety in definition of $Y^n$ which needs to be handled. We have defined $Y$ to be the $W$-associated process, where $W$ is conservative. According to the definition of conservativity, see (1.20), $Y$ then exists as a random element of $C[0,\infty)$, as $A_W^{-1}(B(M_W^{-1}(t)))$ is well defined for all $t$. The same is not necessarily true of $Y^n$, as we do not assume $W^n$ is conservative. Thus in the case $W^n$ is not conservative, we introduce a cemetery point $\Lambda$ and consider $Y^n$ to be taking values in $\mathbb{R} \cup\{\Lambda\}$, where $Y^n_t=\Lambda$ whenever $B(M_{W^n}^{-1}(t)) \notin \text{ran}(A_{W^n})$. This is handled in exactly the same way as in Section 2 of [13], and we refer the reader there for the details.

We are going to show that the ${W}^n$-associated process - ${Y}^n$ - is in some sense `close' to the random walk $\widetilde{S}^n$. It is relatively easy to show that the $W^n$-associated processes converge weakly to the $W$-associated process. The main difficulty is to pass from convergence of these continuous time processes to convergence of the random walks $\widetilde{S}^n$. 

The first step of the proof is to show weak convergence of the sequence $\{Y^n\}_{n\geq1}$ to $Y$. 

\begin{lemma}
    The sequence $\{Y^n\}_{n \geq 1}$ converges weakly in $C_{\mathbb{R}\cup\{\Lambda\}}[0,\infty)$ to $Y$.
\end{lemma}
\begin{proof}
    We refer to [13] for the details.

    Theorem 5 in [13] provides a computational proof, based on definition (1.19), showing that if $\{Z^n\}_{n\geq1}$ is a sequence of deterministic, piecewise constant potentials converging in $D(-\infty,\infty)$ to a deterministic, conservative potential $Z$, then the $Z^n$-associated processes converge almost surely in $C_{\mathbb{R} \cup \{\Lambda\}}[0,\infty)$ to the $Z$-associated process.
    
    By Skorohod's representation theorem, we then switch to a probability space in which the potentials $W^n\to W$ a.s., and apply Theorem 5 in [13]. Since the convergence has then been shown for every realisation of the potentials, we conclude that the result holds for the random potentials as well.
\end{proof}

We now begin to make clear in what sense the process $Y^n$ and the random walk $\widetilde{S}^n$ are `close'. 
\begin{lemma}
    Let $\rho^n_i=\inf\{t \geq 0: Y^n_t \notin (\frac{i-1}{n},\frac{i+1}{n})\},$ where $Y^n$ is the $W^n$-associated process, started from $\frac{i}{n}$.
    Then $\mathbb{P}(Y^n_{\rho^n_i}=\frac{i+1}{n})=\beta^{(n)}_i$ .
\end{lemma}

\begin{proof}
    Since the full proof can be found in Theorem 11 in [13], we omit the details. Using the fact that $W^n$ is piecewise constant, and conditioning on $W^n$, by Lemma 9 and 10 in [13], and Exercise 2.24, Chapter 3 and Exercise 2.16, Chapter 12 in [11], we know 
    \[
    \begin{aligned}
    \mathbb{P}\left(Y^n_{\rho^n_i}=\frac{i+1}{n}\Bigg|W^n\right)&=\frac{e^{-W^n(i)}}{e^{-W^n(i)}+e^{-W^n(i-1)}} \\
    &=\frac{1}{1+e^{W^n(i)-W^n(i-1)}} \\
    &=\frac{1}{1+\left(\frac{1-\beta_i}{\beta_i}\right)^\frac{1}{D_n}} \\
    &=\beta^{(n)}_i,
    \end{aligned}
    \]
    where the last equality follows from the definition of $\beta^{(n)}_i$ in (1.14).
\end{proof}

We define the process
\begin{equation}
V^n_t=Y^n_{\sigma(t)}, \text{ where } \sigma(t)=\sup \{s \leq t: Y^n_s \in \frac{1}{n}\mathbb{Z}\}.
\end{equation}
Intuitively, $V^n_t$ keeps track of $Y^n_t$ on the grid $\frac{1}{n}\mathbb{Z}$. It is also clear that weak convergence of the sequence $\{V^n\}_{n\geq1}$ to $Y$ follows from weak convergence of the sequence $\{Y^n\}_{n\geq1}$ to $Y$.

The relationship between $V^n$ and $\widetilde{S}^n$ is also fairly intuitive, and is made precise in the next Lemma: 

\begin{lemma}
    $\{V^n_t \; ; \; t \geq 0\}$ is equal in law to $\{U^n_t\; ; \; t \geq 0\},$ where $U^n_t=\widetilde{S}^n_{{m}/{n^2}}$ for $t \in [R^n_m,R^n_{m+1}),$ and where $R^n_m=\sum_{k=1}^{m}\rho^n_k$, for $\rho_i$ defined in Lemma 2.2.
\end{lemma}
\begin{proof}
    See Theorem 11 in [13].
\end{proof}

\begin{remark}
    It is shown in [13] that the $(\rho_i^n)_{i \in \mathbb{Z}}$ are i.i.d. with $\mathbb{E}[\rho^n_i]=\frac{1}{n^2}$.
\end{remark}

We have now shown weak convergence of the sequence $\{U^n\}_{n\geq1}$ to $Y$. Thus all that remains to be shown, according to Remark 2.1, is that the $R^n_m$ in Lemma 2.3 can be replaced by their mean. This is done in Section 6 of [13] by first showing tightness of the sequence $\{\widetilde{S}^n\}_{n\geq1}$, and secondly that the finite dimensional distributions of $|\widetilde{S}^n_t-U^n_t|$ converge in probability to 0. The proofs for our model are exactly the same, and so we can show:

\begin{proposition}
    The sequence $\{\widetilde{S}^n\}_{n \geq 1}$ converges weakly in $D[0,\infty)$ to $Y$, the $W$-associated process. Moreover $Y$ is continuous.
\end{proposition}

Having shown the first of the three statements that make up Theorem~1.1, we now turn our attention to the branching processes $\widetilde{M}^n$, defined in (1.18). 

\begin{proposition}
    The sequence $\{\widetilde{M}^{n}(t)\; ; \; t \geq 0\}_{n\geq1} \Rightarrow \{H(t)\; ; \; t \geq 0\}$ in $D[0,\infty)$, where $H(t)=e^{-W(t)}\eta(\tau^{-1}(t))$ for, 
    \[\begin{aligned}
        &\eta(t), \text{ a diffusion satisfying $\eta(0)=1$ and }  d\eta(t)=\sqrt{\eta}dB(t), \\
        & \text{$\tau(t)$ satisfying }t=2\int_0^{\tau(t)}e^{W(s)}ds,
    \end{aligned}\]
    and where $\eta$ is independent of $(W,\tau)$.
\end{proposition}
\begin{proof}
    Since $\widetilde{M}^n$ is a branching process, by the same argument as in [12], we are justified in letting
\begin{equation}
\widetilde{M}^n(s)=\frac{1}{n}\sum_{j=1}^{n\widetilde{M}^n(\lfloor ns \rfloor -1)}N^n_{\lfloor ns \rfloor,j},
\end{equation}

where 
\begin{equation}
    \begin{aligned}
        \mathbb{P}(N^n_{i,j}=k|\beta^{(n)}_i)&=(\beta^{(n)}_i)^k(1-\beta_i^{(n)}), \\
        \mathbb{E}[N^n_{i,j}|\beta^{(n)}_i]&=\frac{\beta_i^{(n)}}{1-\beta_i^{(n)}}=:m^n_i, \\
        \mathbb{E}[(N^n_{i,j}-m^n_i)^2|\beta^{(n)}_i]&=\frac{\beta_i^{(n)}}{(1-\beta_i^{(n)})^2}.
    \end{aligned}
\end{equation}
    
    According to Theorem 2.13 in [3], to show weak convergence of the sequence of branching processes $\{\widetilde{M}^n\}_{n\geq1}$, it suffices to show the following three conditions for $N^n_{i,j}$ defined in (2.4):
\[\begin{aligned}
    &1) \text{ The sequence}\quad\left\{\sum_{i=1}^{\lfloor nt \rfloor}\ln(m^n_i)\; ; \; t\geq0\right\}_{n\geq1} \text{ converges weakly to some limiting process,} \\
    &2) \text{ The sequence } \quad\left\{\frac{1}{n}\sum_{i=1}^{\lfloor nt \rfloor}\mathbb{E}[(N^n_{i,j}-m^n_i)^2|\beta^{(n)}_i]\; ; \; t\geq0 \right\}_{n\geq1} \\ &\quad \text{converges weakly to some absolutely continuous process,} \\
    &3) \text{ The sequence}\quad \frac{1}{n^\frac{3}{2}}\sum_{i=1}^{\lfloor nt \rfloor}\mathbb{E}\left[\Bigg|\frac{N^n_{i,j}}{m^n_i}-1\Bigg|^3\Bigg|\beta^{(n)}_i\right] \text{ converges in probability to 0 for every t.}
\end{aligned}\]

1) follows from our assumption (1.13), and the limit is $-W$. 

To show 2):
\begin{equation}
    \begin{aligned}
    \frac{1}{n}\sum_{i=1}^{\lfloor nt \rfloor}\mathbb{E}\left[(N^n_{i,j}-m^n_i)^2|\beta^{(n)}_i\right]&=\frac{1}{n}\sum_{i=1}^{\lfloor nt \rfloor}\frac{\beta_i^{(n)}}{(1-\beta_i^{(n)})^2} \\
    &=:V(t)
    \end{aligned}
\end{equation}
Now observe that, according to the definition of $\beta^{(n)}_i$ in (1.14), we have
\[
\beta^{(n)}_i=\frac{1}{1+(\frac{1-\beta_i}{\beta_i})^{\frac{1}{D_n}}}.
\]
By assumption 1 in (1.12) we see that $\beta^{(n)}_i\to \frac{1}{2}$ a.s. and uniformly in $i$ as $n\to \infty$. Hence we immediately obtain that $V(t)\to 2t$ a.s. as $n\to \infty$.

To show 3): 

\[
\begin{aligned}
\frac{1}{n^\frac{3}{2}}\sum_{i=1}^{\lfloor nt \rfloor}\mathbb{E}\left[\Bigg|\frac{N^n_{i}}{m^n_i}-1\Bigg|^3\Bigg|   \beta^{(n)}_i\right]&=\frac{1}{n^\frac{3}{2}}\sum_{i=1}^{\lfloor nt \rfloor}\mathbb{E}\left[\Bigg|\frac{N^n_{i}-m^n_i}{m^n_i}\Bigg|^3\Bigg|\beta^{(n)}_i\right] \\
&=\frac{1}{n^\frac{3}{2}}\sum_{i=1}^{\lfloor nt \rfloor}\frac{1}{({m^n_i})^{3}}\mathbb{E}\left[|N^n_{i}-m^n_i|^3|\beta^{(n)}_i\right] \\
&=\frac{1}{n^\frac{3}{2}}\sum_{i=1}^{\lfloor nt \rfloor}\frac{1}{({m^n_i})^{3}}\frac{(2-\beta_i^{(n)})\beta_i^{(n)}}{(1-\beta_i^{(n)})^3} \\
&=\frac{1}{n^\frac{3}{2}}\sum_{i=1}^{\lfloor nt \rfloor}\frac{2-\beta_i^{(n)}}{{\beta_i^{(n)}}^2} \\
\end{aligned}
\]

Again by noting that $\beta^{(n)}_i$ are uniformly bounded in $i$, we see that the condition is satisfied.

By Theorem 2.13 in [3], we are done.
\end{proof}
The final step in the proof of Theorem~1.1 is to identify the process $\{H(t)\; ; \; t\geq0\}$ as the local time of $Y$.
\begin{proposition}
    The process $\{H(t)\; ; \;t \geq 0\}$ defined in Proposition 2.2 is equal in law to $\{L_{Y}(a,M_W(\psi))\; ; \; a \geq 0 \}$, where $L_Y(x,t)$ is the local time at level $x$ by time t of $Y$, $dM_W$ is the speed measure defined by $W$, and $\psi=\inf\{t \geq 0:l(0,t)>1\}$ for $l(x,t)$, the local time of Brownian motion.
\end{proposition}
\begin{proof}
    Essentially the same proof in the i.i.d. environment case is carried out in [2]. We provide it here for completeness.

    Since $\eta(x)$ has $\eta(0)=1$ and solves
    \begin{equation}
        d\eta(x)=\sqrt{\eta}dB(x),
    \end{equation}
    we know $\{\frac{1}{2}\eta(2x);x \geq 0\}$ solves the same stochastic differential equation and starts from $\frac{1}{2}$. Therefore $\widetilde{\eta}(x):=\eta({2x})$ satisfies $\widetilde{\eta}(0)=1$ and
    \begin{equation}
        d\widetilde{\eta}(x)=\sqrt{2}\sqrt{\widetilde{\eta}}dB(x).
    \end{equation}
    Hence by the classical Ray-Knight Theorem, we know 
    \begin{equation}
        \{l(x,\psi)\; ; \; x \geq 0\}\stackrel{d}{=}\{\widetilde{\eta}(x)\; ; \; x \geq 0\}.
    \end{equation}
We now find an expression for $L_Y(x,t)$. Fix $x \geq 0.$ By the definition of local time, it is enough to check $\int_0^t1_{\{-x \leq Y_s \leq x\}}ds=\int_{-x}^xL_Y(y,t)dy$. Recall the definition of $Y$ as the $W$-associated process, according to (1.19), and that $A_W$ and $dM_W$ denote the scale function and speed measure defined by $W$.
    \[\begin{aligned}
        \int_0^t1_{\{-x \leq Y_s \leq x\}}ds&=:\int_0^tf(Y_s)ds \\
        &=\int_0^tf(A_W^{-1}(B(M_W^{-1}(s))))ds \\
        &=\int_0^{M_W^{-1}(t)}f(A_W^{-1}(B(u)))\exp(-2W(A_W^{-1}(B(u))))du \\
        &= \int_{-\infty}^{+\infty}f(A_W^{-1}(x))\exp(-2W(A_W^{-1}(x)))l(x,M_W^{-1}(t))dx \\
        &=\int_{-\infty}^{+\infty}f(x)\exp(-2W(x))l(A_W(x),M_W^{-1}(t))e^{W(x)}dx \\
        &=\int_{-\infty}^{+\infty}f(x)e^{-W(x)}l(A_W(x),M_W^{-1}(t))dx \\
        &= \int_{-x}^{x}f(y)e^{-W(y)}l(A_W(y),M_W^{-1}(t))dy
    \end{aligned}\]
    Hence
    \[
    \{L_Y(x,t)\; ; \; x \geq 0\}=\{e^{-W(x)}l(A_W(x),M_W^{-1}(t))\; ; \; x \geq 0 \}.
    \]
    We now make the key observation concerning the time change $\tau$ appearing in the definition of $H(t)$ in Proposition 2.2. 
    \begin{equation}
    \tau^{-1}(x)=2\int_0^xe^{W(s)}ds=2A_W(x).
    \end{equation}
    It is this equality that ultimately connects the branching process and random walk in the way implied by Theorem~1.1. We then immediately obtain,
    \[
    \begin{aligned}
        \{L_Y(x,M_W(\psi))\; ; \; x \geq 0\}&=\{e^{-W(x)}l(A_W(x),\psi)\; ; \; &x \geq 0\} \\
        &\stackrel{d}{=}\{e^{-W(x)}\widetilde{\eta}(A_W(x))\; ; \; &x \geq 0\} \\
        &=\{e^{-W(x)}\eta(2A_W(x))\; ; \; &x \geq 0\} \\
        &=\{e^{-W(x)}\eta(\tau^{-1}(x))\; ; \; &x \geq 0\} \\
        &=\{H(x)\; ; \; &x \geq 0\}.
    \end{aligned}
    \]
\end{proof}

\textbf{Proof of Theorem 1.1:}

The convergence of $\{\widetilde{S}^n\}_{n\geq1}$ to $Y$, the $W$-associated process, was shown in Proposition 2.1. The convergence of $\{\widetilde{M}^n\}_{n\geq1}$ to the process $H$ was shown Proposition 2.2, and the identification of $H$ as the local time of the process $Y$ was obtained in Proposition 2.3.

We conclude this section with some remarks on possible extensions of Theorem~1.1.

As mentioned in the introduction, Theorem~1.1 extends the connection between random walks in random environments and branching processes in random environments to include dependent environments. What we have shown here is that as soon as a RWRE admits a scaling limit, the geometric branching process defined from the RWRE also admits a scaling limit, and the two limits satsify a Ray-Knight Theorem.
   
However, there is nothing in the proof of Theorem~1.1 particular to the geometric distribution. Indeed, we could attempt to connect a branching process in a random environment with a random walk in a random environment through local time in their respective scaling limits. To do so we would employ the same techniques used in the above two sections, and in particular rely on the following three ingredients in the proof: 
    
A) A branching process in a (arbitrary) random environment that satisfies conditions 1,2 and 3 from [3]. In particular there must be some limiting potential $\widetilde{W}$ satisfying 1).

B) A random walk in a random environment that converges in distribution to a Brownian motion in the random potential $-\widetilde{W}$.

C) The process $V(t)$ defined in (2.5) to be of the form $at$, for some deterministic constant $a$. This allows for the identification (observed in (2.9)) of $\widetilde{\tau}^{-1}(x)$ and ${A_{-\widetilde{W}}}$, the scale function defined by the potential $-\widetilde{W}$.

In our paper we started with condition B and proved conditions A and C. Starting with condition A and proving condition B (under the mild extra assumption that the mean of the offspring distribution at generation $i$ in the $n$th rescaling converges uniformly in $i$ to 1) follows exactly the same procedure we outlined in the proof of weak convergence of $\{\widetilde{S}^n\}_{n\geq1}$. Condition C does not follow automatically from condition A in the case of arbitrary offspring distribution. It does follow if the variance of the offspring distribution is a continuous function of the mean.

If the process $V(t)$ is not of the form $at$, it would be interesting to know whether the resulting scaling limit of the branching process cannot be the local time of a Brownian motion in a random potential.

\section{Proof of Theorem 1.3}
In this section we shall prove Theorem~1.3. We must define some extra notation.

Let $\mathcal{M}_R(\mathbb{R})$ denote the space of non-negative Radon measures on $\mathbb{R}$, with the vague topology. We recall a sequence $\mu_n \in \mathcal{M}_R(\mathbb{R})$ converges in the vague topology to $\mu$, if  
\[
\int_{\mathbb{R}} fd\mu_n \to \int_{\mathbb{R}}fd\mu 
\]
for all $f\in C_c(\mathbb{R})$, where $C_c(\mathbb{R})$ denotes the space of all continuous functions on $\mathbb{R}$ with compact support. The space $\mathcal{M}_R(\mathbb{R})$ with the vague topology is a Polish space.

Recall $\mathcal{W}$ is the space of stopped paths, and that the sequence of Brownian snakes $(\mathbb{W}^n)_{n\geq1}$ is a therefore a sequence of random elements of $\mathcal{W}$. 

The first main result of this section is the following proposition:
\begin{proposition}
    The process $(\mathbb{W}^n,L^{n,\cdot},T^{n,0}_\cdot)$ converges weakly in $D_{\mathcal{W} \times \mathcal{M}_R(\mathbb{R})} \times \mathcal{M}(\mathbb{R})$ to a unique limit $(\mathbb{W},L^{\cdot},T^{0}_\cdot)$. Moreover, $L^a_t$ is the local time of $S$ at level $a$ by time $t$, and $T^0_r$ is the inverse local time of $S$ at level $0$ by time $r$. 
\end{proposition}

The proof of Theorem 1.3 will follow from Proposition 3.1. Recall the convention that we use the notation $T^{n,0}_{\cdot}$ to denote both the increasing function $r \mapsto T^{n,0}_r$ and the corresponding measure.

The proof of Proposition 3.1 is long and we indicate the structure. We start with Lemma 3.1, which shows that the sequence of inverse local times is tight in $\mathcal{M}(\mathbb{R})$, along with some other properties of the inverse local times which shall be needed later. Section 3.1 is then devoted to estimates on asymptotics for survival probabilities for branching processes in our random environments. These estimates are needed to prove C-tightness of the snake processes, which is the content of Section 3.2. In Section 3.2 we also prove the sequence of snake processes has a unique limit. Section 3.3 deals with C-tightness of the local time processes. The characterisation of the limit points of the sequence of local time and inverse local time processes, and the proof of Proposition 3.1 and Theorem 1.3, are then obtained in Section 3.4.

\begin{lemma}
    \[
    \begin{aligned}
        &\mathbf{a)}\text{ For any $r>0$, the sequence of random variables $\{T^{n,0}_r\}_{n \geq 1}$ is tight,} \\ 
        &\text{\quad and any limiting point }T^0_r \text{ is a.s. not equal to 0.} \\
        &\mathbf{b)}\text{ For any $\epsilon>0, A>0$ there exists $R>0$ such that } \mathbb{P}(T^{n,0}_R>A)\geq1-\epsilon. \\
        &\mathbf{c)}\text{ The sequence $\{T^{n,0}_{\cdot}\}_{n \geq 1}$ is tight in $\mathcal{M}(\mathbb{R})$.} \\
        &\mathbf{d)}\text{ For any limiting point $T^0_{\cdot}$ of $\{T^{n,0}_{\cdot}\}_{n \geq 1}$ } \\
        &\text{ \quad and for any fixed $r \in \mathbb{R}_+$, $T^0_{t}$ is a.s. continuous at $t=r$.}
    \end{aligned}
    \]
\end{lemma}

\begin{remark}
    The identification of the unique limit of the sequence $\{T^{n,0}_{\cdot}\}_{n \geq 1}$ will be shown in Lemma 3.11 (c).
\end{remark}

Lemma 3.1 is an exact replica of Lemma 3.9 in [8], and the proofs for our model are exactly the same.\footnote{In place of Lemma 2.3 (b) in [2] we appeal to the form of the weak limit of sequence of total mass processes. We carry out an analagous but more complicated argument in the proof of Lemma 3.10.} Therefore, we refer the reader to Lemma 3.9 in [8] for the proof, and turn to convergence of the snake process. In order to show convergence of the snake process we first need to prepare some estimates.

\subsection{Estimates on asymptotics for branching processes in random environments}
In this section we present variants on three results which are stated and proved in [8] for the model considered there. Our results and the proofs thereof are slightly different.

We start with a result on branching processes in fixed environments which is well known. It is proved in [8].

\begin{lemma}
    Let $\{M^n\}_{n\geq 1}$ be a sequence branching process with geometric offspring distribution with parameter $\frac{1}{2}-\frac{b_n}{4n}$. 
    
    Assume $\lim\limits_{n \to \infty}b_n=b$ and $M^n_0=1$ for all $n$. For any $\delta>0$ define
    \[h(b,\delta)= \begin{cases}
        \frac{1}{\delta}, \text{ if $b=0$} \\
        \frac{b}{1-e^{-b\delta}}, \text{ else.}
    \end{cases}
    \]
    Then
    \[
    \lim_{n \to \infty}n\mathbb{P}(M^n_{\lfloor n\delta \rfloor}>0)=h(b,\delta)
    \]
\end{lemma}
\begin{proof}
    See Lemma 2.1 in [8].
\end{proof}

We make the following observations. According to the definition of $\beta^{(n)}_i$ in (1.14),
\begin{equation}
    \beta_i^n=\frac{1}{2}-\frac{1}{4D_n}\ln\left(\frac{1-\beta_i}{\beta_i}\right)+o\left(\frac{1}{D_n^3}\right),
\end{equation}
for $D_n$ introduced in (1.13). Since, by (1.13),
\[
\left\{\frac{1}{D_n}\sum_{i=1}^{\lfloor nt \rfloor}\ln\left(\frac{1-\beta_i}{\beta_i}\right)\; ; \; t\geq0 \right\} \Rightarrow \{W(t)\; ; \; t\geq0\}
\]
in $D[0,\infty)$, a result due to Lamperti in [5] tells us that, necessarily,
\begin{equation}
    D_n^2 \sim n^{2H}L(n)
\end{equation}
for some $0<H<1$ and $L$ slowly varying. We recall $L$ is a slowly varying function if $L(n)\sim L(an)$, for all $a>0$. Lastly, by 1) in (1.24), we see that $\frac{1}{2}<H<1$ for the model we are considering.

The importance of the form (3.1) is that it allows for a version of Lemma 3.2 to be true for the random environment we are considering. Let $\widetilde{M}^n$ be a branching process in the random environment $(\beta^{(n)}_i)_{i\geq1}$ (the branching mechanism is then defined as in (2.4)) such that $\widetilde{M}^n(0)=1$ for all n. Recall the offspring distribution of $\widetilde{M}^n$ at generation $i$ is then geometric with parameter $1-\beta^{(n)}_i$.

\begin{lemma}
    Fix $\delta>0$, then
    \[
    \limsup\limits_{n \to \infty}n\mathbb{P}(\widetilde{M}^n(\lfloor n\delta \rfloor)>0) \leq h(b,\delta)
    \]
\end{lemma}
\begin{proof}
    Let $s \in [0,1]$. By the branching property,
    \[
    \mathbb{E}[s^{\widetilde{M}^n(k)}]=\mathbb{E}[f_{\beta^{(n)}_1}\circ f_{\beta^{(n)}_2}\circ ...\circ f_{\beta^{(n)}_k}(s)],
    \]
    where $f_{\beta^{(n)}_i}=\frac{1-\beta^{(n)}_i}{1-\beta^{(n)}_is}$, the probability generating function of the offspring distribution at generation $i$, conditioned on the environment. We will derive an expression for the above generating function, and compare it with that of the $M^n(k)$ defined in Lemma 3.2.

        We first fix $s$ and let 
        \[
        \mathbb{E}[f_{\beta^{(n)}_1}\circ f_{\beta^{(n)}_2}\circ ...\circ f_{\beta^{(n)}_k}(s)]=:\mathbb{E}[G(\boldsymbol{\beta}^{(n)})]
        \]
        where $G$ is a function of the vector $\boldsymbol{\beta}^{(n)}$. Taylor expanding $G$ around $(\frac{1}{2},...,\frac{1}{2})$, we see
        \[
        G(\boldsymbol{\beta}^{(n)})=G(\mathbf{\frac{1}{2}})+\sum_{i=1}^k\partial_{\beta^{(n)}_i}G|_{\mathbf{\frac{1}{2}}}\delta_i^{(n)}+\frac{1}{2}\sum_{i,j=1}^k\partial^2_{\beta^{(n)}_i,\beta^{(n)}_j}G|_{\mathbf{\frac{1}{2}}}\delta_i^{(n)}\delta^{(n)}_j+\text{higher order terms}.
        \]
        where, by (3.1),
        \begin{equation}
        \delta^{(n)}_i:=\beta^{(n)}_i-\frac{1}{2}=-\frac{1}{4D_n}\ln\left(\frac{1-\beta_i}{\beta_i}\right)+o\left(\frac{1}{D_n^3}\right).
        \end{equation}
        Hence we have
        \[
        \mathbb{E}[s^{\widetilde{M}^n(k)}]=\mathbb{E}\left[G(\mathbf{\frac{1}{2}})+\sum_{i=1}^k\partial_{\beta^{(n)}_i}G|_{\mathbf{\frac{1}{2}}}\delta_i^{(n)}+\frac{1}{2}\sum_{i,j=1}^k\partial^2_{\beta^{(n)}_i,\beta^{(n)}_j}G|_{\mathbf{\frac{1}{2}}}\delta_i^{(n)}\delta^{(n)}_j\right] + \text{ higher order terms.}
        \]
        It is a simple but tedious exercise in the chain rule to verify that the coefficients of $\delta^{(n)}_i$ and $\delta^{(n)}_i\delta^{(n)}_j$ are bounded above and below independently of k. This follows from the form of $f_{\beta^{(n)}_i}$ and that $s \in [0,1]$. We note that the lower bound may be negative. 
        
        By assumption 2) in (1.12) the leading term in the definition of $\delta^{(n)}_i$ vanishes in expectation. Hence,
        \begin{equation}
        \begin{aligned}
        \mathbb{E}[s^{\widetilde{M}^n(k)}] &\geq G(\mathbf{\frac{1}{2}})+C\sum_{i,j=1}^k\mathbb{E}[\delta_i^n\delta^n_j]+\text{higher order terms} \\
        &=G(\mathbf{\frac{1}{2}})+C'\mathbb{E}\left[\left(\sum_{i=1}^k\delta^n_i\right)^2\right] + \text{higher order terms} \\
        &=G(\mathbf{\frac{1}{2}})+\bar{C}\frac{1}{D_n^2}\mathbb{E}\left[\left(\sum_{i=1}^k\ln\left(\frac{1-\beta_i}{\beta_i}\right)\right)^2\right]+\text{higher order terms}
        \end{aligned}
        \end{equation}
        for some universal constants $C, C', \bar{C} \in \mathbb{R}$. A completely analogous but much simpler computation allows one to obtain (recall $s$ is still arbitrary but fixed)
        \begin{equation}
        \mathbb{E}[s^{M^n(k)}]\leq G(\mathbf{\frac{1}{2}}) + \widehat{C}k\frac{b_n}{n} + \text{higher order terms,}
        \end{equation}
        where again we do not stipulate the sign of $\widehat{C}$.
        
        Now, if we let $k= \lfloor n\delta \rfloor$, then by (3.4) and (3.5),
        \[
        \mathbb{E}[s^{M^n(k)}] \leq \mathbb{E}[s^{\widetilde{M}^n(k)}]
        \]
        holds if and only if 
        \begin{equation}
        \widehat{C}\delta b_n + O\left(\frac{1}{n}\right)\leq \bar{C}\frac{1}{D_n^2}\mathbb{E}\left[\left(\sum_{i=1}^{\lfloor n\delta \rfloor}\ln\left(\frac{1-\beta_i^n}{\beta_i^n}\right)\right)^2\right]+o\left(\frac{1}{D_n^2}\right).
        \end{equation}
        We can ignore the higher order terms on the right hand side of (3.6) according to assumption 3) in (1.24) and Remark (1.3). Moreover, since $\delta$ is fixed, the remaining term is asymptotically finite according to assumption 2) in (1.24).
        
        Hence choosing $b_n \to b \in \mathbb{R}$ such that the inequality (3.6) holds, and taking the corresponding $M^n$ as defined in Lemma 3.2, we see
        \[
        \mathbb{E}[s^{M^n(\lfloor n\delta \rfloor)}] \leq \mathbb{E}[s^{\widetilde{M}^n(\lfloor n\delta \rfloor)}]
        \]
        for all $n$ sufficiently large, and thus 
        \[
        \mathbb{P}(\widetilde{M}^n(\lfloor n\delta \rfloor)>0) \leq \mathbb{P}({M}^n(\lfloor n\delta \rfloor)>0),
        \]
        for all $n$ sufficiently large. Therefore,
        \[
        \limsup_{n \to \infty}\mathbb{P}(\widetilde{M}^n(\lfloor n\delta \rfloor)>0) \leq \limsup_{n \to \infty}\mathbb{P}({M}^n(\lfloor n\delta \rfloor)>0)=h(b,\delta)
        \]
        and we are done.
\end{proof}

\begin{lemma}
    Let $\widetilde{M}^n$ be as above and fix $\delta >0$. Then,
    \[
    \limsup_{n \to \infty}\mathbb{E}[\widetilde{M}^n(\lfloor n\delta \rfloor)] \leq 1+C(\delta)
    \]
    for some $C$ independent of $n$.
\end{lemma}
\begin{proof}
    By the branching property and the fact that $\widetilde{M}^n(0)=1$,
    \[
    \mathbb{E}[\widetilde{M}^n(k)]=\mathbb{E}\left[\prod_{i=1}^k\frac{\beta_i^{(n)}}{1-\beta_i^{(n)}}\right]
    \]
    According to the definition of $\beta^{(n)}_i$ in (1.14),
    \[
    \begin{aligned}
        \prod_{i=1}^k\frac{\beta_i^{(n)}}{1-\beta_i^{(n)}}&=\exp\left[{-\frac{1}{D_n}\sum_{i=1}^k\ln\left(\frac{1-\beta_i}{\beta_i}\right)}\right] \\
        &=1-\frac{1}{D_n}\sum_{i=1}^k\ln\left(\frac{1-\beta_i}{\beta_i}\right)+\frac{1}{2D_n^2}\left(\sum_{i=1}^k\ln\left(\frac{1-\beta_i}{\beta_i}\right)\right)^2 + o\left(\frac{1}{D_n^3}\right),
    \end{aligned}
    \]
    and so again by assumption 2) in (1.12), 
    \[
    \mathbb{E}[\widetilde{M}^n(k)]=1+\frac{1}{2D_n^2}\mathbb{E}\left[\left(\sum_{i=1}^k\ln\left(\frac{1-\beta_i}{\beta_i}\right)\right)^2\right],
    \]
    holds where we have omitted the higher order terms by assumption 3) in (1.24) and Remark 1.3. Setting $k=\lfloor n\delta \rfloor$ and recalling assumption 2) in (1.24) gives the desired result.
\end{proof}

\begin{lemma}
    Let $X^n$ be defined as in (1.23). Let $f$ be a bounded, non-negative measurable function and let $\delta>0$. Then
    \[
    \mathbb{E}[X^n_{\lfloor n\delta \rfloor/n}(f)] \leq X^n_0(P_{\delta}f)(1+C(\delta))
    \]
    for some $C$ independent of $n$, where $(P_t)_{t\geq0}$ is the heat semigroup.
\end{lemma}
\begin{proof}
    Let $N^n_{i,k}$ denote the number of offspring of the $i$th individual at time $t=\frac{k}{n}$. That is, the particle is born at time $t=\frac{k-1}{n}$ and reproduces a random number of offspring at time $t=\frac{k}{n}$ with distribution Geom($1-\beta^n_k$). 

    Let $\mathcal{U}_{i,k-1}({\frac{j}{n}})$ be the spatial position of this particle at time $t=\frac{j}{n}$, $k-1 \leq j \leq k$. Let $m^n_k=\mathbb{E}[N^n_{i,k}]=\frac{\beta^n_k}{1-\beta^n_k}$. Let $\widetilde{M}^n(k)$ be the number of particles alive at time $t=\frac{k}{n}$. Then,
    \[
    X^n_{{k}/{n}}(f)=\frac{1}{n}\sum_{i=1}^{\widetilde{M}^n(k-1)}N^n_{i,k}f\left(\mathcal{U}_{i,k-1}\left(\frac{k}{n}\right)\right)
    \]
    Hence
    \[
    \begin{aligned}
    \mathbb{E}[X^n_{{k}/{n}}(f)]&=\mathbb{E}\left[\frac{1}{n}\sum_{i=1}^{\widetilde{M}^n(k-1)}N^n_{i,k}f\left(\mathcal{U}_{i,k-1}(\frac{k}{n})\right)\right] \\
    &=\mathbb{E}\left[\mathbb{E}\left[\frac{1}{n}\sum_{i=1}^{\widetilde{M}^n(k-1)}N^n_{i,k}f\left(\mathcal{U}_{i,k-1}(\frac{k}{n})\right)\Bigg|\beta^{(n)}_k,X^n_{({k-1})/{n}}\right]\right] \\
    &=\mathbb{E}\left[m^n_k\mathbb{E}\left[\frac{1}{n}\sum_{i=1}^{\widetilde{M}^n(k-1)}f\left(\mathcal{U}_{i,k-1}(\frac{k}{n})\right)\Bigg|\beta^{(n)}_k,X^n_{({k-1})/{n}}\right]\right] \\
    &=\mathbb{E}[m^n_kX^n_{({k-1})/{n}}\left(P_{{1}/{n}}(f)\right)]
    \end{aligned}
    \]
    where in the third equality we used the independence of the spatial motion and the environment.
    Iterating this relation we see,
    \[
    \begin{aligned}
    \mathbb{E}[X^n_{{k}/{n}}(f)]&=\mathbb{E}\left[\prod_{i=1}^km^n_iX^n_0\left(P_{{k}/{n}}(f)\right)\right] \\
    &=X^n_0\left(P_{{k}/{n}}(f)\right)\mathbb{E}\left[\prod_{i=1}^km^n_i\right] \\
    &\leq X^n_0\left(P_{{k}/{n}}(f)\right)\left(1+C(k)\right).
    \end{aligned}
    \]
    As in the proof of Lemma 3.4, in the last line we again used assumptions 2 and 3 in (1.24) and Remark 1.3 to obtain the constant $C(k)$ independent of $n$. Setting $k=\lfloor n\delta \rfloor$ gives the desired result.
\end{proof}

\subsection{The snake process}

 We turn our attention to showing C-tightness of the sequence $\{\mathbb{W}^n\}_{n\geq1}$ in $D_{\mathcal{W}}[0,\infty)$. That is, we wish to show for any $T>0, \text{ and }a>0$,
 \begin{equation}
     \lim_{\delta_1 \to 0}\limsup_{n \to \infty}\mathbb{P}(\sup_{0 \leq t \leq T}\sup_{\delta\leq\delta_1}d(\mathbb{W}^n_{t+\delta},\mathbb{W}^n_t)>a)=0
 \end{equation}
 where $d$ is the metric on the space $\mathcal{W}$ defined in Section 1.4.

 According to the definition of $d$, we need to control 
 \begin{equation}
 \begin{aligned}
 &\mathbb{P}(\sup_{0 \leq t \leq T}\sup_{\delta\leq\delta_1}|S^n_{t+\delta}-S^n_t|>a), \text{ and} \\
 &\mathbb{P}(\sup_{0 \leq t \leq T}\sup_{\delta\leq\delta_1}\sup_{s \geq0}|\mathbf{W}^n_{t+\delta}(s)-\mathbf{W}^n_t(s)|>a).
 \end{aligned}
 \end{equation}

 We shall show C-tightness of the contour processes $\{S^n\}_{n\geq1}$ in $D[0,\infty)$, and control the second of these two terms. Tightness of the sequence of snake processes will follow.

 \begin{proposition}
     The sequence $\{S^n\}_{n\geq1}$ is C-tight in $D[0,\infty)$. Moreover, the sequence converges to a unique limit, $S$.
 \end{proposition}
 \begin{proof}
     Recall $S^n_t=f(\widetilde{S}^n)(t)$, where $f$ is defined in (1.25). Recall weak convergence of $\{\widetilde{S}^n\}_{n\geq1}$ to $Y$, the $W$-associated process was shown in Proposition 2.1. Thus it suffices to show $f$ is a continuous function in the Skorohod topology.

    First note that for any $x\in D[0,\infty)$,
    \[
    f(x)(t)=g \circ x(t),
    \]
    where
    \[
    g:\mathbb{R}\to \mathbb{R}, \; y \mapsto K-(|y|-K).
    \]
    Continuity of $f$ now follows from (Lipschitz) continuity of $g$, and the definition of convergence in the Skorohod topology.

    Hence the sequence $\{S^n\}_{n\geq1}$ converges weakly in $D[0,\infty)$ to $S:=f(Y)$, the $W$-associated process reflected at 0 and $K$. Moreover, $t\mapsto S_t$ is continuous and hence the sequence $\{S^n\}_{n\geq1}$ is C-tight.
 \end{proof}
 
 To control the second of the two terms in (3.8) we need to bound the maximal displacement of a particle from its ancestor on a time interval of length $\delta$.

 Having shown Lemmas 3.3 and 3.5 for our model (which are analogues of Lemmas 2.2, and 2.4 in [8]) we can bound the maximal displacement of a particle from its ancestor on a time interval of length $\delta$ in exactly the same way as in [8].

 Fix $\eta \in (0,\frac{1}{4})$. Let 
 \[
 \begin{aligned}
 Z^{n,\eta}_{a,\delta}= \#\{&\text{particles alive at time $a_n+\delta$ that are displaced by more than $\delta^{\frac{1}{2}-\eta}$} \\ &\text{ from the ancestor at time $a_n$}\},
 \end{aligned}
\]
where $a_n=\frac{\lfloor an \rfloor}{n}$. Using Lemmas 3.3 and 3.5 in place of Lemmas 2.2 and 2.4 in [8] and appealing to exactly the same proof, we can show
\begin{lemma}
     For any $\epsilon>0$ there exists $\delta_1>0$ such that,
     \[
     \limsup_{n \to \infty}\mathbb{P}(\sup_{a \leq K}\sup_{\delta < \delta_1}Z^{n,2\eta}_{a,\delta}>0) \leq \epsilon
     \]
\end{lemma}
\begin{proof}
See Lemma 4.4 in [8].
\end{proof}

We then obtain as an immediate corollary, for arbitrary $r>0$,

\begin{lemma}
    For any $\epsilon > 0$ there exists $\delta_1$ such that
    \[
    \limsup_{n\to \infty}\mathbb{P}(\sup_{t \leq T^{n,0}_r}\sup_{\delta \leq \delta_1}\sup_{s \leq (S^n_t-\delta)_{+}}|\mathbf{W}^n_t(s+\delta)-\mathbf{W}^n_t(s)|>\delta_1^{\frac{1}{2}-2\eta})<\epsilon.
    \]
\end{lemma}

We can now prove C-tightness of the sequence $\{\mathbb{W}^n\}_{n\geq1}$.

\begin{lemma}
    The sequence of processes $\{\mathbb{W}^n\}_{n \geq 1}$ is C-tight in $D_{\mathcal{W}}[0,\infty)$.
\end{lemma}
\begin{proof}
    C-tightness of the contour processes was shown in Proposition 3.2. Hence we only need to control
    \[
    \mathbb{P}(\sup_{0 \leq t \leq T}\sup_{\delta\leq\delta_1}\sup_{s \geq0}|\mathbf{W}^n_{t+\delta}(s)-\mathbf{W}^n_t(s)|>a).
    \]
    Fix arbitrary $\alpha>0$ and let $a=\alpha^{\frac{1}{2}-2\eta}$. We have
    \[
    \begin{aligned}
    &\mathbb{P}(\sup_{0 \leq t \leq T^{n,0}_r}\sup_{\delta\leq\delta_1}\sup_{s \geq0}|\mathbf{W}^n_{t+\delta}(s)-\mathbf{W}^n_t(s)|>a) \leq \\ &\mathbb{P}(\sup_{0 \leq t \leq T^{n,0}_r}\sup_{\delta\leq\delta_1}|S^n_{t+\delta}-S^n_t|>\alpha) \\ &+\mathbb{P}(\sup_{0 \leq t \leq T^{n,0}_r}\sup_{\delta \leq \alpha}\sup_{s \leq (S^n_t-\delta)_{+}}|\mathbf{W}^n_t(s+\delta)-\mathbf{W}^n_t(s)|>a)
    \end{aligned}
    \]
    The result now follows from C-tightness of the contour processes, Lemma 3.7 and Lemma 3.1 (b).
\end{proof}

Since we know the sequence of contour processes are not only C-tight, but also converge weakly to a unique limiting process $S$ (see Proposition 3.2), we immediately obtain

\begin{lemma}
    The sequence $\{\mathbb{W}^n\}_{n \geq 1}$ converges weakly in $D_{\mathcal{W}}[0,\infty)$ to a unique limiting process $\mathbb{W}$.
\end{lemma}
\begin{proof}
    This follows from the construction of the $\mathbf{W}^n$ in (1.27), the mutual independence between the Brownian paths $B_1,B_2,...$ defined in (1.27) and the contour process $S^n$, and the uniqueness of the limit $S$.
\end{proof}

We now show convergence of the local time processes.

\subsection{The local time process}

We will show the sequence $\{L^{n,\cdot}\}_{n \geq 1}$ is C-tight in $D_{\mathcal{M}_R(\mathbb{R})}$. We follow the same approach taken in [8]. In contrast to the model considered in [8], here we can exploit the form of the limiting branching processes in random environments for our model (see Proposition 2.2).

\begin{lemma}
    The sequence $\{L^{n,\cdot}\}_{n \geq 1}$ is C-tight in $D_{\mathcal{M}_R(\mathbb{R_+})}$.
\end{lemma}
\begin{proof}

Let $\bar{L}^{n,a}_t=L^{n,a}_{t \wedge T^{n,0}_r}$, for arbitrary $r>0$. By Lemma 3.1 (b), it suffices to show C-tightness of $\bar{L}^{n,\cdot}$ in $D_{\mathcal{M}(\mathbb{R})}[0,\infty)$ (recall properties of convergence in the vague topology).  As is noted in [8], since for each $a$ and $n$, the function $s \mapsto \bar{L}^{n,a}_s$ is non-decreasing, it suffices to prove tightness of $\bar{L}^{n,\cdot}_t$ for each fixed $t$.

Hence we have established it suffices to show 
\[
\lim_{h \to 0}\limsup_{n \to \infty}\mathbb{P}(\sup_{0 \leq a \leq T}|\bar{L}^{n,a+h}_t-\bar{L}^{n,a}_t|>\epsilon)=0
\]
for arbitrary $T>0$ and each fixed $t$.

We now define, for any $i,j,\delta>0$,
\begin{equation}
\bar{Z}^{n,i,j,\delta}_s=\bar{L}^{n,i\delta+s}_{T^{n,i\delta}_{(j+1)\delta}}-\bar{L}^{n,i\delta}_{T^{n,i\delta}_{j\delta}}.
\end{equation}

On the event $t<T^{n,0}_r$ we have 
\[
\begin{aligned}
&\sup_{0 \leq a \leq T}|\bar{L}^{n,a+h}_t-\bar{L}^{n,a}_t| \leq \\
&\sup_{i\delta \leq T}\sup_{j\delta\leq r}\sup_{m \in [0,\delta]}|\bar{L}^{n,i\delta+h+m}_{T^{n,i\delta}_{j\delta}}-\bar{L}_{T^{n,i\delta}_{j\delta}}^{n,i\delta+m}|+\sup_{i\delta \leq T}\sup_{j\delta \leq r}\sup_{s \leq \delta}\bar{Z}^{n,i,j,\delta}_s.
\end{aligned}
\]

Fix $\delta>0$. For each fixed $i$ and $j$ such that $i\delta \leq T, j\delta \leq r$, $\bar{L}^{n,i\delta+\cdot}_{T^{n,i\delta}_{j\delta}}$ can be seen to be the Feller rescaled branching process induced by considering the trees underneath the excursions of $S^n_t$ above the level $i\delta$, where $\lfloor j\delta n \rfloor$ many excursions are considered. As such, by exactly the same argument as in Proposition 2.2, we have that the sequence $\{\bar{L}^{n,i\delta+\cdot}_{T^{n,i\delta}_{j\delta}}\}_{n \geq 1}$ is C-tight.

Hence clearly, 
\[
\lim_{h \to 0}\limsup_{n \to \infty}\mathbb{P}(\sup_{m \in [0,\delta]}|\bar{L}^{n,i\delta+h+m}_{T^{n,i\delta}_{j\delta}}-\bar{L}_{T^{n,i\delta}_{j\delta}}^{n,i\delta+m}|>\epsilon)=0
\]

And so for any $\delta>0$ fixed we have,
\[
\lim_{h \to 0}\limsup_{n \to \infty}\mathbb{P}(\sup_{i\delta \leq T}\sup_{j\delta\leq r}\sup_{m \in [0,\delta]}|\bar{L}^{n,i\delta+h+m}_{T^{n,i\delta}_{j\delta}}-\bar{L}_{T^{n,i\delta}_{j\delta}}^{n,i\delta+m}|>\epsilon)=0.
\]

We turn our attention to the $\bar{Z}^{n,i,j,\delta}$. For each fixed $i$ and $j$ such that $i\delta \leq T, j\delta \leq r$, $\bar{Z}^{n,i,j,\delta}_{s}$ is a Feller rescaled branching process in the random environment $(\beta^{(n)}_k)_{k \geq {\lfloor i\delta  \rfloor}}$, starting from `time' $i\delta$, with $\bar{Z}^{n,i,j,\delta}_{0}=\delta$ for all $n$. We are going to employ exactly the same approach as in Proposition 2.2 to show each of these sequences $\{\bar{Z}^{n,i,j,\delta}\}_{n\geq1}$ converges weakly to a limiting process, and use properties of these limiting processes to control the probability that $\bar{Z}^{n,i,j,\delta}$ is large. 

Condition 1 of Theorem 2.13 in [3] is easily seen to be met. Indeed, it is immediate that 
\[
\begin{aligned}
\left\{\sum_{k={\lfloor i\delta \rfloor}}^{\lfloor nt \rfloor+{\lfloor i\delta \rfloor}}\ln\left(\frac{\beta_k^n}{1-\beta^n_k}\right)\; ; \; t \geq 0 \right\}_{n\geq1} &\Rightarrow \{W_{i,\delta}(t)\; ; \; t \geq 0\}
\end{aligned}
\]
where 
\begin{equation}
\begin{aligned}
\{W_{i,\delta}(t)\; ; \; t \geq 0\} &\stackrel{d}{=}\{-W(t+i\delta)-(-W(i\delta))\; ; \; t \geq 0\} \\
&=\{W(i\delta)-W(t+i\delta)\; ; \; t\geq0\} \\
&=:\{\bar{W}_{i,\delta}(t)\; ; \; t\geq0\}
\end{aligned}
\end{equation}
for $W$ defined in (1.13). Hence condition 1 is satisfied. (Recall $\{\sum_{k=1}^{\lfloor nt \rfloor}\ln(m^n_k);t \geq 0 \}$ converges weakly to $\{-W(t);t \geq 0\}$, for $m^n_k$ defined in (2.4)). Moreover, if $N^n_k$ is distributed as $\text{Geom}(1-\beta^{(n)}_k)$, then
\[
\frac{1}{n}\sum_{k={\lfloor i\delta \rfloor}}^{\lfloor nt \rfloor+{\lfloor i\delta \rfloor}}\mathbb{E}[(N^n_k-m^n_k)^2|\beta^{(n)}_k]\to 2t
\]
a.s., by the same argument as in Proposition 2.2. The third condition from Theorem 2.13 in [3] can also be seen to be satisfied. Hence we see 
\[
\begin{aligned}
\{\bar{Z}^{n,i,j,\delta}_t\; ; \; t \geq 0\}_{n\geq1} &\Rightarrow \{e^{W_{i,\delta}(t)}\eta_{\delta}(\tau_{i,\delta}^{-1}(t))\; ; \; t\geq0\},
\end{aligned}
\]
where $\eta_\delta$, the Feller diffusion defined in Proposition 2.2 and started at $\delta$, is independent of $({W_{i,\delta}},\tau_{i,\delta})$, and $\tau_{i,\delta}(t)$ solves
\[
t=2\int_0^{\tau_{i,\delta}(t)}e^{-W_{i,\delta}(s)}ds.
\]

Now, due to this independence, and the definition of $\bar{W}_{i,\delta}$ in (3.7), we have 
\begin{equation}
\begin{aligned}
\{e^{W_{i,\delta}(t)}\eta_{\delta}(\tau_{i,\delta}^{-1}(t))\; ; \; t\geq0\} &\stackrel{d}{=}\{e^{\bar{W}_{i,\delta}(t)}\bar{\eta}_{\delta}(\bar{\tau}_{i,\delta}^{-1}(t))\; ; \; t\geq0\} \\
&=:\{Z^{i,\delta}(t)\; ; \; t\geq0 \},
\end{aligned}
\end{equation}
where $\bar{\tau}_{i,\delta}(t)$ solves
\[
t=2\int_0^{\bar{\tau}_{i,\delta}(t)}e^{-\bar{W}_{i,\delta}(s)}ds.
\]
and $(\bar{W}_{i,\delta},\bar{\tau}_{i,\delta})$ is independent of $\bar{\eta}$, and $\bar{\eta}$ is an independent copy of $\eta$.

Having established the convergence of $\{\bar{Z}^{n,i,j,\delta}\}_{n\geq1}$ to $Z^{i,\delta}$, we can appeal to the Portmanteau theorem to show that for each fixed $i$,
\begin{equation}
\begin{aligned}
\limsup_{n \to \infty}\mathbb{P}(\sup_{j\delta\leq r}\sup_{s \leq \delta}\bar{Z}^{n,i,j,\delta}_s\geq \epsilon)&\leq \mathbb{P}(\sup_{j\delta\leq r}\sup_{s \leq \delta}Z^{i,\delta}(s)\geq \epsilon) \\
&\leq \mathbb{P}(\sup_{i\delta\leq T}\sup_{s \leq \delta}Z^{i,\delta}(s)\geq \epsilon),
\end{aligned}
\end{equation}

where we used that the sequence $\{\bar{Z}^{n,i,j,\delta}\}_{n\geq1}$ and its weak limit are both independent of $j$. 

It is clear that each $\bar{W}_{i,\delta}$ is equal to the potential $-W$ started from time $i\delta$. Hence 
\[
\sup_{i\delta \leq T}\sup_{s \leq \delta}(\bar{W}_{i,\delta}(s))\leq \sup_{s\leq T+\delta}-W(s).
\]
Moreover, by the same observation, 
\[
\bar{\tau}^{-1}_{i,\delta}(t)\leq2\int_0^t\inf_{s\leq T+\delta}e^{-W(s)}ds=:\bar{\tau}_{max}(t)
\]
From these observations and the form of the limiting $Z^{i,\delta}$ in (3.11), we see,
\[
\begin{aligned}
    \mathbb{P}(\sup_{i\delta\leq T}\sup_{s \leq \delta}Z^{i,\delta}(s)\geq \epsilon) &\leq \mathbb{P}(\sup_{s\leq T+\delta}e^{W(s)}\sup_{s\leq \bar{\tau}_{max}(\delta)}\bar{\eta}_{\delta}(s)\geq \epsilon)
\end{aligned}
\]
and hence by (3.12),
\[
\begin{aligned}
\limsup_{n \to \infty}\mathbb{P}(\sup_{i\delta \leq T}\sup_{j\delta\leq r}\sup_{s \leq \delta}\bar{Z}^{n,i,j,\delta}_s\geq \epsilon) &\leq \mathbb{P}(\sup_{s\leq T+\delta}e^{W(s)}\sup_{s\leq \bar{\tau}_{max}(\delta)}\bar{\eta}_{\delta}(s)\geq \epsilon)
\end{aligned}
\]
also.

If we now choose $\delta_k\to0$, and appeal to the continuous mapping theorem applied to $(W,\bar{\tau}_{max},\bar{\eta}_{\delta})$, we see that
\[
\mathbb{P}(\sup_{s\leq T+\delta_k}e^{W(s)}\sup_{s\leq \bar{\tau}_{max}(\delta_k)}\eta_{\delta_k}(s)\geq \epsilon)\to0 \text{ as $\delta_k\to0.$}
\]

Hence we have shown 
\[
\lim_{\delta\to 0}\limsup_{n \to \infty}\mathbb{P}(\sup_{i\delta \leq T}\sup_{j\delta\leq r}\sup_{s \leq \delta}\bar{Z}^{n,i,j,\delta}_s\geq \epsilon)=0
\]
and we are done.
\end{proof}

\begin{corollary}
    The sequence $(\mathbb{W}^n,L^{n,\cdot},T^{n,0}_{\cdot})_{n\geq1}$ is tight in $D_{\mathcal{W} \times \mathcal{M}_R(\mathbb{R})}\times \mathcal{M}(\mathbb{R})$.
\end{corollary}
\begin{proof}
    Immediate from Lemma 3.1 (c), Lemma 3.8 and Lemma 3.10.
\end{proof}

\subsection{The proof of Proposition 3.1 and Theorem 1.2}

We can now prove Proposition 3.1. In order to do so, we need to identify any limiting point of the sequence of local time processes as the local time process of of $S$, and we need to identify any limiting point of the sequence of inverse local times as the inverse local time of $S$. (Recall the identification of a unique limit of the sequence $\{\mathbb{W}^n\}_{n\geq1}$ was shown in Lemma 3.9). In order to do this we quote an amalgamation of Lemmas 4.11, 4.13, 4.14 in [8], whose proofs are exactly the same for our model. In each case, we sketch the proof and refer to [8] for the details.

\begin{lemma}
\[    
\begin{aligned}
&\mathbf{a})\text{ Let $L^{\cdot}$ be a limiting point of the sequence $\{L^{n,\cdot}\}_{n\geq1}$. Then $L^a_t$ is the local time } \\
&\text{\quad of $S$ at level $a$ by time $t$. } \\
&\mathbf{b}) \text{ The map $t \mapsto L^0_t$ is continuous.} \\
&\mathbf{c}) \text{ Let $T^0_r$ be the any limiting point of } \{T^{n,0}_r\}_{n\geq1}. \\
&\text{\quad Then $T^0_r$ is the inverse local time of $S$ at level 0 by time $r$.}
\end{aligned}
\]
\end{lemma}
\begin{proof}
    According to Corollary 3.1 we may switch to a probability space where $(L^{n,\cdot},T^{n,0}_{\cdot})\to (L^{\cdot},T^0_{\cdot})$ a.s.
    
    (a):
    Fix $a$ and $t$. It suffices to show 
    \[
    \int_0^t1_{S_s\leq a}ds=\int_0^aL^s_tds. \\
    \]
    Using the fact that any $s\neq a$ is a point of continuity of the function $t \mapsto L^a_t$, properties of weak convergence of measures imply that
    \[
    L^{n,r}_s \to L^r_s
    \]
    a.s. for all $s\neq r$.
    From this it immediately follows that 
    \[
    \int_0^aL^{n,r}_tdr \to \int_0^aL^r_tdr \text{\quad a.s.}
    \]
    On the other hand, according to the definition of $L^{n,\cdot}$,
    \[
    \int_0^aL^{n,r}_tdr=\int_0^{{\lfloor n^2t \rfloor}/{n^2}}1_{S^n_s\leq a}ds.
    \]
    Using that $S^n$ converge weakly to $S$, and the tightness of $\{L^{n,r}_t\}_{n\geq1}$ to control the error of $1_{S^n_t\leq a+\delta}$ one can show 
    \[
    \int_0^{{\lfloor n^2t \rfloor}/{n^2}}1_{S^n_s\leq a}ds \to \int_0^t1_{S_s\leq a}ds \text{\quad a.s.}
    \]
    
    (b):
    It suffices to show $\{L^{n,0}_{T^{n,0}_r\wedge \cdot}\}$ is C-tight in $D_{\mathbb{R}}[0,\infty)$ for arbitrary $r>0$.
    We suppose for a contradiction that it is not. That is, we suppose there exist $\epsilon,\epsilon_1$ such that for all $\delta>0,n\geq1$,
    \begin{equation}
    \mathbb{P}(\sup_{t\leq T^{n,0}_r}L^{n,0}_t-L^{n,0}_{t-\delta}>\epsilon)>\epsilon_1
    \end{equation}
    We have the following inclusion,
    \begin{equation}
    \{\sup_{t\leq T^{n,0}_r}L^{n,0}_t-L^{n,0}_{t-\delta}>\epsilon\} \subset \{\exists i=1,2,...,\lfloor {2r}/{\epsilon}\rfloor :T^{n,0}_{{(i+1)\epsilon}/{2}}-T^{n,0}_{{i\epsilon}/{2}}>\delta\}
    \end{equation}
    We control the probability of the event on the right hand side in (3.14) using that $\left(T^{n,0}_{{(i+1)\epsilon}/{2}}-T^{n,0}_{{i\epsilon}/{2}}\right)_{i\geq1}$ are identically distributed and appealing to the second statement in Lemma 3.1 (a). Then we obtain that for all $n$ sufficiently large,
    \[
    \mathbb{P}(\exists i=1,2,...,\lfloor {2r}/{\epsilon}\rfloor :T^{n,0}_{{(i+1)\epsilon}/{2}}-T^{n,0}_{{i\epsilon}/{2}}>\delta)\leq \frac{\epsilon_1}{2},
    \]
    which is in contradiction with (3.13).

    (c):
    By Lemma 3.1 (d), and again using properties of weak convergence of measures, we see that 
    \[
    T^{n,0}_r \to T^0_r \text{\quad a.s.}
    \]
    Fixing $\delta>0$ and observing the definition of the $L^{n,0}_{\cdot}$ gives, 
    \[
    L^{n,0}_{T^{n,0}_{r+\delta}}\geq r+\delta.
    \]
    Continuity of the map $t \mapsto L^0_t$ gives,
    \[
    L^0_{T^0_{r+\delta}}\geq r+\delta,
    \]
    and so 
    \[
    \inf\{s\geq0:L^0_s>r\}\leq T^0_{r+\delta}
    \]
    An analogous argument shows
    \[
    \inf\{s\geq0:L^0_s>r\}\leq T^0_{r-\delta}
    \]
    Since $\delta$ was arbitrary, applying Lemma 3.1 (d) gives the desired result.
\end{proof}

We obtain Proposition 3.1 as an immediate consequence.

\textbf{Proof of Proposition 3.1:}

Immediate from Lemma 3.1 (c), Proposition 3.2, Lemma 3.9, Lemma 3.10, and Lemma 3.11 (a) and (c).

We are now ready to prove Theorem 1.3.

\textbf{Proof of Theorem 1.3:}

The weak convergence of $(\mathbf{W}^n,S^n)$ to $\mathbb{W}$ was shown in Lemma 3.9.

Let $f \in \mathcal{C}_b(\mathbb{R}^d)$, and switch to a probability space where 
\[
((\mathbf{W}^n,S^n),L^n,T^{n,0}) \to ((\mathbf{W},S),L,T^0)
\]
almost surely. 

Note, $T^{n,0}_1 \to T^0_1$ almost surely, by the continuity of $T^0_{r}$ at the point $r=1$. By Lemma 3.11 (b), $L^0_{\cdot}$ is continuous at $T^0_1$. Hence we immediately get
\[
\int_0^{T^{n,0}_1}f(s)L^{n,0}(ds) \to \int_0^{T^0_1}f(s)L^0(ds)
\]
almost surely. 

By Lemma 3.11 (c) we know $T^0_1$ is the inverse local time of $S$ at level 0 by time 1. By Lemma 3.11 (a) we know $L^{n,a}_t$ is the local time of $S$ at level $a$ by time $t$.

Note that C-tightness of the snake processes implies uniform on compacts convergence of $\mathbb{\widehat{W}}^n$ to $\mathbb{\widehat{W}}$. Hence for $t=0$ we obtain Theorem~1.3. For $t>0$, the result follows immediately as $L^t(ds)$ is unchanged at the point $s=T^0_1$. Hence we have shown 
\[
\int_0^{T^{n,0}_1}\phi(\mathbb{\widehat{W}}^n_s)L^{n,t}(ds) \Rightarrow \int_0^{T^{0}_1}\phi(\mathbb{\widehat{W}}_s)L^t(ds).
\]

As was mentioned in the introduction, 
\[
X^n_t(\phi)=\int_0^{T^{n,0}_1}\phi(\mathbb{\hat{W}}^n_s)L^{n,t}(ds),
\]
for $X^n$ defined in (1.23). By Theorem~1.2 (proved in the next section), the sequence $\{X^n\}_{n\geq1}$ is tight in $D_{\mathcal{M}(\mathbb{R}^d)}[0,K]$. Let $X$ be a limiting point. We have shown $X$ has the representation
\[
\forall \phi \in \mathcal{C}_b(\mathbb{R}^d), \forall t \in [0,K], \quad X_t(\phi)=\int_0^{T^{0}_1}\phi(\mathbb{\widehat{W}}_s)L^t(ds),
\]
where $L^a_s$ is the local time of $S$ at level a by time t, and $T^0_1$ is the inverse local time of $S$ at level 0 by time 1. The extension to any $\phi \in \mathcal{B}(\mathbb{R}^d)$ is trivial. The theorem is proved.

\section{The Martingale Problem and Proof of Theorem 1.2}

Here we shall show tightness of the sequence of measure-valued processes defined in (1.23) and derive the martingale problem satisfied by any limit point. We use broadly the same techniques as in [7].

Let $\mathbb{\dot{R}}^d:=\mathbb{R}^d\cup \{\infty\}$ denote the one point compactification of $\mathbb{R}^d$. We will first show C-tightness of the sequence $\{X^n\}_{n\geq1}$ in $D_{\mathcal{M}(\mathbb{\dot{R}}^d)}[0,\infty)$, where it suffices to show C-tightness of $\{X^n(\phi)\}_{n\geq1}$ in $D[0,\infty)$ for all $\phi \in \mathcal{C}^2_b(\mathbb{\dot{R}}^d)$. The identification of any limiting point as being a random element of $D_{\mathcal{M}(\mathbb{R}^d)}[0,K]$ that vanishes at $K$ and is continuous on $[0,K)$ will then follow from Theorem 1.2.

We define notation to keep track of the particles in exactly the same way as in [7]. Throughout this section, unless stated otherwise, $\phi$ denotes a function in $\mathcal{C}^2_b(\mathbb{\dot{R}}^d)$.

Recall $X^n_0=\delta_0$, a point mass located at the origin in $\mathbb{R}^d$. We append a cemetery state $\Lambda$ to the state space $\mathbb{\dot{R}}^d$, and adopt the convention that $\phi(\Lambda)=0$ for all $\phi:\mathbb{\dot{R}}^d \to \mathbb{R}$.

Define the family
\[
I=\{\alpha=(\alpha_0,\alpha_1,...,\alpha_N):\alpha_i\in \mathbb{N}_0 \text{ for $0\leq i \leq N$},N\in \mathbb{N}\}.
\]
$\alpha \in I$ keeps track of a single lineage of $\alpha_0$, one of the initial particles located at the origin.

Define the length of $\alpha$ by $|\alpha|=N$, and write $\alpha \sim_n t$ when $\frac{|\alpha|}{n}\leq t\leq \frac{|\alpha|+1}{n}$. Adopt the convention $\alpha-j=(\alpha_0,...,\alpha_{N-j})$, where $|\alpha|=N$.

We now turn to defining the spatial motion of the particles.

Define $\{\widetilde{B}^{n,\alpha}:|\alpha|=0\}$ to be a collection of independent Brownian motions in $\mathbb{R}^d$, stopped at time $t=\frac{1}{n}$, started from the origin. 

Define now recursively $\{\widetilde{B}^{n,\alpha}:|\alpha|=N\}$ to be a collection of conditionally independent Brownian motions given $\sigma(\{\widetilde{B}^{n,\alpha}: \alpha_0 \leq n, |\alpha|=N-1\})$, with 
\[
\widetilde{B}^{n,\alpha}_t=\widetilde{B}^{n,\alpha -1}_t, \quad \text{when $t\leq \frac{|\alpha|}{n}$}.
\]
We adopt the convention that $\infty\in \mathbb{\dot{R}}^d$ is an absorbing state for the $\{\widetilde{B}^{n,\alpha}\}_{\alpha \in I,n\geq1}$.

We now define the branching. Let $\{N^{n,\alpha}:\alpha \in I\}$ be a collection random variables such that $\{N^{n,\alpha}:|\alpha|=i\}$ are conditionally independent of each other and of the spatial motion, given $\beta^{(n)}$. The law of $N^{n,\alpha}$ for $|\alpha|=i$ is geometric with parameter $1-\beta^{(n)}_{i+1}$. That is,
\[
\mathbb{P}(N^{n,\alpha}=k|\beta^{(n)}_{i+1})=(\beta^{(n)}_{i+1})^k(1-\beta^{(n)}_{i+1}).
\]

To make the branching compatible with the spatial motion we introduce the stopping times
\[
\tau^{n,\alpha}=\begin{cases}
0, \quad \text{if $\alpha_0>n$} \\
\min_{0\leq i\leq |\alpha|}\{\frac{i+1}{n}:N^{\alpha|_i}=0\}, \quad \text{if this set $\neq \emptyset$}, \\
\frac{|\alpha|+1}{n}, \quad \text{else.}
\end{cases}
\]

and set 
\[
B^{n,\alpha}_t=\begin{cases}
    \widetilde{B}^{n,\alpha}_t, \quad \text{if $t<\tau^{n,\alpha}$} \\
    \Lambda, \quad \text{else.}
\end{cases}
\]

We can now define the measure-valued process in the compactified state space. 
\begin{equation}
X^n_t(A):=\frac{\#\{B^{n,\alpha}_t\in A: \alpha \sim_n t\}}{n}, \quad A\in\mathcal{B}(\mathbb{\dot{R}}^d).
\end{equation}

We introduce a filtration for the process $X^n$:
\begin{equation}
\begin{aligned}
&\mathcal{F}^n_t:=\sigma\left(\{\widetilde{B}^{n,\alpha},N^{n,\alpha}:|\alpha|<i\}\right)\vee \left[\cap_{s \geq t}\sigma(\widetilde{B}^{n,\alpha}_s:|\alpha|=i+1)\right], \\
&\text{when $\frac{i}{n}\leq t < \frac{i+1}{n}$,}
\end{aligned}
\end{equation}

\begin{remark}
    Let $a_n=\frac{1}{n}$, $s_n=\frac{\lfloor ns \rfloor}{n}$, and $\alpha\sim_n s_n$. Then $N^{n,\alpha}$ is $\mathcal{F}^n_{s_n+a_n}$ measurable, and
    \[
    \mathbb{E}[N^{n,\alpha}|\mathcal{F}^n_{s_n+a_n}]=\frac{\beta^{(n)}_{\lfloor sn \rfloor +1}}{1-\beta^{(n)}_{\lfloor sn \rfloor +1}}.
    \]
\end{remark}

Let $\alpha\sim_n s$ for $s \in [i_n,i_n+a_n)$ and let $\phi \in \mathcal{C}^2_b(\mathbb{\dot{R}}^d)$. Define the martingale noise associated with the spatial motion of the particle $\alpha$ over the time interval $[i_n,i_n+a_n)$ by,
\[
M^{n,\alpha,i_n}_s(\phi)=\begin{cases}
    \phi(\widetilde{B}^{n,\alpha}_s)-\phi(\widetilde{B}^{n,\alpha}_{i_n})-\frac{1}{2}\int_{i_n}^s\Delta \phi (\widetilde{B}^{n,\alpha}_r)dr, \quad \text{if $\widetilde{B}^{n,\alpha}_{i_n}\neq \Lambda$} \\
    0 \quad \text{else.}
\end{cases}
\]

To show C-tightness of $\{X^n(\phi)\}_{n\geq1}$ it will be crucial to identify a useful semimartingale decomposition of $X^n(\phi)$. We first give a simpler decomposition. 

For the $X^n$ defined in (4.1) and $\phi \in \mathcal{C}^2_b(\mathbb{\dot{R}}^d)$, we have
\begin{equation}
    X^n_t(\phi)=X^n_0(\phi)+\frac{1}{2}\int_0^tX^n_s(\Delta \phi)ds+\widehat{M}^n_t(\phi), 
\end{equation}
where
\begin{equation}
\begin{aligned}
    \widehat{M}^n_t(\phi)=&\frac{1}{n}\sum_{s_n<t_n}\sum_{\alpha \sim_n s_n}\phi(B^{n,\alpha}_{s_n+a_n})(N^{n,\alpha}-1)+\frac{1}{n}\sum_{s_n<t_n}\sum_{\alpha \sim_n s_n}M^{n,\alpha,s_n}_{s_n+a_n}(\phi) \\
    &+\frac{1}{n}\sum_{\alpha \sim_n t_n}M^{n,\alpha,t_n}_t(\phi).
    \end{aligned}
\end{equation}
The equation (4.4) can be seen to be true, however $\widehat{M}^n_t(\phi)$ is not an $\mathcal{F}^n_t$ martingale. To make it into a martingale we make the further decomposition given by:
\[
\begin{aligned}
N^{n,\alpha}-1 \mapsto &N^{n,\alpha}-\mathbb{E}[N^{n,\alpha}|\beta^{(n)},\mathcal{F}^n_{s_n}]+ \\
&\mathbb{E}[N^{n,\alpha}|\beta^{(n)},\mathcal{F}^n_{s_n}]-\mathbb{E}[N^{n,\alpha}|\mathcal{F}^n_{s_n}]+\mathbb{E}[N^{n,\alpha}|\mathcal{F}^n_{s_n}]-1
\end{aligned}
\]
for $\alpha \sim_n s_n$.

Substituting this decomposition into the definition of $\widehat{M}^n_t(\phi)$ in (4.4), we see that we will obtain martingale noise from three distinct sources. One due to the spatial motion of the particles, another from the branching conditioned on the environment, and a third due to the fluctuations of the environment. We will also be left with an aggregate drift term. More formally,
\begin{equation}
    X^n_t(\phi)=X^n_0(\phi)+\frac{1}{2}\int_0^tX^n_s(\Delta \phi)ds+Z^n_t(\phi)+N^n_t(\phi)+M^n_t(\phi)+A^n_t(\phi)
\end{equation}
where
\begin{equation}
    \begin{aligned}
        &Z^n_t(\phi)=\frac{1}{n}\sum_{s_n<t_n}\sum_{\alpha \sim_n s_n}\phi(B^{n,\alpha}_{s_n+a_n})(N^{n,\alpha}-\mathbb{E}[N^{n,\alpha}|\beta^{(n)},\mathcal{F}^n_{s_n}]), \\
        &N^n_t(\phi)=\frac{1}{n}\sum_{s_n<t_n}\sum_{\alpha \sim_n s_n}\phi(B^{n,\alpha}_{s_n+a_n})(\mathbb{E}[N^{n,\alpha}|\beta^{(n)},\mathcal{F}^n_{s_n}]-\mathbb{E}[N^{n,\alpha}|\mathcal{F}^n_{s_n}]), \\
        &M^n_t(\phi)=\frac{1}{n}\sum_{s_n<t_n}\sum_{\alpha \sim_n s_n}M^{n,\alpha,s_n}_{s_n+a_n}(\phi)+\frac{1}{n}\sum_{\alpha \sim_n t_n}M^{n,\alpha,t_n}_t(\phi), \\
        &A^n_t(\phi)=\frac{1}{n}\sum_{s_n<t_n}\sum_{\alpha \sim_n s_n}\phi(B^{n,\alpha}_{s_n+a_n})(\mathbb{E}[N^{n,\alpha}|\mathcal{F}^n_{s_n}]-1).
    \end{aligned}
\end{equation}
It is clear that $Z^n_t(\phi),N^n_t(\phi),M^n_t(\phi)$ are all $\mathcal{F}^n_t$ martingales, closely related to the branching conditioned on the environment; the fluctuations due to the environment; and the spatial motion, respectively. $A^n_t(\phi)$ is the aggregate drift term. To show C-tightness of $X^n_t(\phi)$, it suffices to show C-tightness of all terms in the decomposition (4.6). 

In the tightness proofs we shall make indispensable use of the fact that 
\begin{equation}
\lim_{k\to \infty}\limsup_{n\to \infty}\mathbb{P}(\sup_{t\leq T}X^n(1)>k)=0,
\end{equation}
for any fixed $T>0$. This follows from the C-tightness of the sequence of total mass processes, which in turn follows from the more general result in Proposition 2.2.

\begin{lemma}
    For any $t\geq0$, $\sup_{s\leq t}|M^n_t(\phi)| \stackrel{\mathbb{P}}{\to}0$.
\end{lemma}
\begin{proof}
    We show the statement when $t=t_n$ for $t_n$ defined in Remark 4.1. This suffices. First note,
    \begin{equation}
    \begin{aligned}
    \langle M^n(\phi) \rangle_{t_n}&=\frac{1}{n^2}\sum_{s_n<t_n}\mathbb{E}\left[\left(\sum_{\alpha \sim_n s_n}M^{n,\alpha,s_n}_{s_n+a_n}(\phi)\right)^2\Bigg|\mathcal{F}^n_{s_n}\right] \\
    &=\frac{1}{n^2}\sum_{s_n<t_n}\sum_{\alpha \sim_n s_n}\mathbb{E}[M^{n,\alpha,s_n}_{s_n+a_n}(\phi)^2|\mathcal{F}^n_{s_n}],
    \end{aligned}
    \end{equation}
    where in the second equality we used the fact that given $\mathcal{F}^n_{s_n}$, the $\{M^{n,\alpha,s_n}_{s_n+a_n}(\phi)\}_{\alpha\sim_n s_n}$ are independent, mean 0 random variables.

    Noting that Ito's formula implies that $M^{n,\alpha,s_n}_{s_n+a_n}(\phi)$ is a stochastic integral of $\nabla \phi$ against $B^{n,\alpha}$, and applying Ito's isometry, gives
    \begin{equation}
        \mathbb{E}[M^{n,\alpha,s_n}_{s_n+a_n}(\phi)^2|\mathcal{F}^n_{s_n}]\leq \frac{|| \nabla\phi ||^2_{\infty}}{n}.
    \end{equation}
    And hence
    \begin{equation}
        \begin{aligned}
            \langle M^n(\phi) \rangle_{t_n}&\leq \frac{1}{n^3}\sum_{s_n<t_n}\sum_{\alpha \sim_n s_n}||\nabla \phi ||^2_{\infty} \\
            &=\frac{1}{n^2}\sum_{s_n< t_n}||\nabla \phi ||^2_{\infty}X^n_{s_n}(1) \\
            &\leq \frac{1}{n^2}(nt||\nabla \phi ||^2_{\infty}\sup_{s_n< t_n}X^n_{s_n}(1)) \\
            &=\frac{1}{n}(t||\nabla \phi ||^2_{\infty}\sup_{s_n< t_n}X^n_{s_n}(1))
        \end{aligned}
    \end{equation}
    From (4.7) we see that 
    \[
    \langle M^n(\phi) \rangle_{t_n} \stackrel{\mathbb{P}}{\to}0.
    \]
    By Lenglart's inequality, for arbitrary $\epsilon, \delta >0$,
    \[
    \mathbb{P}(\sup_{s_n < t_n}|M^n_t(\phi)|>\epsilon)\leq \mathbb{P}(\langle M^n(\phi) \rangle _{t_n} \geq \delta ) + \frac{\delta}{\epsilon^2}.
    \]
    Letting $n\to \infty$ and $\delta \to 0$ concludes the proof.
\end{proof}

\begin{lemma}
    The sequence $\{\frac{1}{2}\int_0^tX^n_s(\Delta \phi) ds:t\geq0\}_{n\geq 1}$ is C-tight.
\end{lemma}
\begin{proof}
    The proof is essentially the same as in [7]. Note
    \[
    \begin{aligned}
    \Bigg|\frac{1}{2}\int_0^tX^n_r(\Delta \phi)dr - \frac{1}{2}\int_0^sX^n_r(\Delta \phi)dr\Bigg|&\leq \frac{1}{2}|t-s|\sup_{r\leq t}X^n_r(\Delta \phi) \\
    &\leq ||\frac{1}{2}\Delta \phi ||_{\infty}\sup_{r\leq t}X^n_r(1)
    \end{aligned}
    \]
    Observing (4.7) yields the result.
\end{proof}

Before showing C-tightness of the terms $Z^n(\phi), N^n(\phi)$, and $A^n(\phi)$, we first note that the process
\[
\frac{1}{n}\sum_{s_n<t_n}\sum_{\alpha \sim_n s_n}\phi(B^{n,\alpha}_{s_n+a_n})(N^{n,\alpha}-1)
\]
is the branching process representing the total mass of the measure-valued process defined in (4.1), with a uniformly bounded spatial term multiplying each generation. As such the proof of C-tightness of the branching process (see Proposition 2.2) carries over, and we immediately obtain 
\begin{lemma}
    The process $\{\frac{1}{n}\sum_{s_n<t_n}\sum_{\alpha \sim_n s_n}\phi(B^{n,\alpha}_{s_n+a_n})(N^{n,\alpha}-1)\}\; ; \; t\geq0\}$ is C-tight.
\end{lemma}

In view of Lemma 4.3, C-tightness of $Z^n(\phi)$ will follow from C-tightness of $N^n(\phi)$ and $A^n(\phi).$

\begin{lemma}
    For any $t>0,$ $\sup_{s\leq t}|N^n_s(\phi)|\stackrel{\mathbb{P}}\to0$.
\end{lemma}
\begin{proof}
    The proof is similar to that of Lemma 4.1.
    First note 
    \[
    \begin{aligned}
    &\langle N^n(\phi)\rangle _{t_n}= \\
    &\frac{1}{n^2}\sum_{s_n< t_n}\mathbb{E}\left[\left(\sum_{\alpha \sim_n s_n}\phi(B^{n,\alpha}_{s_n+a_n})\left(\mathbb{E}[N^{n,\alpha}|\beta^{(n)},\mathcal{F}^n_{s_n}]-\mathbb{E}[N^{n,\alpha}|\mathcal{F}^n_{s_n}]\right)\right)^2\Bigg|\mathcal{F}^n_{s_n}\right] \\
    &\leq||\phi^2||_{\infty}\sup_{s_n<t_n}X^n_{s_n}(1)^2\sum_{s_n< t_n}\mathbb{E}\left[\left(\frac{\beta^{(n)}_{\lfloor sn \rfloor +1}}{1-\beta^{(n)}_{\lfloor sn \rfloor +1}}-\mathbb{E}\left[\frac{\beta^{(n)}_{\lfloor sn \rfloor +1}}{1-\beta^{(n)}_{\lfloor sn \rfloor +1}}\Bigg|\mathcal{F}^n_{s_n}\right]\right)^2\Bigg|\mathcal{F}^n_{s_n}\right]
    \end{aligned}
    \]
    where we used Remark 4.1 to eliminate the dependence on $\alpha$ when moving from the second to the third line.
    
    We now show 
    \[
    \sum_{s_n< t_n}\mathbb{E}\left[\left(\frac{\beta^{(n)}_{\lfloor sn \rfloor +1}}{1-\beta^{(n)}_{\lfloor sn \rfloor +1}}-\mathbb{E}\left[\frac{\beta^{(n)}_{\lfloor sn \rfloor +1}}{1-\beta^{(n)}_{\lfloor sn \rfloor +1}}|\mathcal{F}^n_{s_n}\right]\right)^2\Bigg|\mathcal{F}^n_{s_n}\right] \to 0
    \]
    a.s.
    
    By the formula (1.38) and the uniform-in-$i$ boundedness of the $\beta_i$ (see assumption 1 in (1.12)), we have
    \[
    \begin{aligned}
    \Bigg|\frac{\beta^{(n)}_{\lfloor sn \rfloor+1}}{1-\beta^{(n)}_{\lfloor sn \rfloor +1}}-\mathbb{E}\left[\frac{\beta^{(n)}_{\lfloor sn \rfloor+1}}{1-\beta^{(n)}_{\lfloor sn \rfloor+1}}\Bigg|\mathcal{F}^n_{s_n}\right]\Bigg|
    \leq \frac{C}{D_n}+o\left(\frac{1}{D_n^2}\right),
    \end{aligned}
    \]
    for some universal constant C. As such,
    \[
    \sum_{s_n< t_n}\mathbb{E}\left[\left(\frac{\beta^{(n)}_{\lfloor sn \rfloor +1}}{1-\beta^{(n)}_{\lfloor sn \rfloor +1}}-\mathbb{E}\left[\frac{\beta^{(n)}_{\lfloor sn \rfloor +1}}{1-\beta^{(n)}_{\lfloor sn \rfloor +1}}|\mathcal{F}^n_{s_n}\right]\right)^2\Bigg|\mathcal{F}^n_{s_n}\right] \leq tn\left(\frac{C}{D_n}\right)^2\to 0
    \]
    due to the fact that $D_n^2\sim n^{2H}L(n)$ for some $\frac{1}{2}<H<1$ (see (3.2)). The result follows by the continuous mapping theorem and the same argument involving Lenglart's inequality as in Lemma 4.1.
\end{proof}

\begin{remark}
    From this proof it also follows that the sequence of $\mathcal{F}^n_t$ martingales
    \[
    \left\{\frac{1}{n}\sum_{s_n<t_n}\sum_{\alpha\sim_n s_n}\mathbb{E}\left[N^{n,\alpha}|\beta^{(n)},\mathcal{F}^n_{s_n}\right]-\mathbb{E}[N^{n,\alpha}|\mathcal{F}^n_{s_n}] \; ; \; t\geq0 \right\}_{n\geq1}
    \]
    converge weakly to the 0 process.
\end{remark}

\begin{lemma}
    The sequence $\{A^n_t(\phi):t\geq0\}_{n\geq1}$ is C-tight.
\end{lemma}
\begin{proof}
    By Remark 4.1 we have 
    \[
    |A^n_t(\phi)-A^n_s(\phi)|\leq ||\phi||_{\infty}\sup_{s_n<r_n\leq t_n}X^n_{r_n}(1)\Bigg|\sum_{s_n<r_n\leq t_n}\mathbb{E}\left[\frac{\beta^{(n)}_{\lfloor rn \rfloor +1}}{1-\beta^{(n)}_{\lfloor rn \rfloor +1}}\Bigg|\mathcal{F}^n_{r_n}\right]-1\Bigg|
    \]
    and so, by (4.7) and the continuous mapping theorem, C-tightness of $\{A^n_t(\phi)\; ; \; t\geq0\}_{n\geq1}$ will follow from C-tightness of the sequence
    \[
    \left\{\sum_{s_n<t_n}\mathbb{E}\left[\frac{\beta^{(n)}_{\lfloor sn \rfloor +1}}{1-\beta^{(n)}_{\lfloor sn \rfloor +1}}\Bigg|\mathcal{F}^n_{s_n}\right]-1\; ; \; t\geq0\right\}_{n\geq1}
    \]
    We will first show C-tightness of the sequence
    \[
    \left\{\sum_{s_n<t_n}\frac{\beta^{(n)}_{\lfloor sn \rfloor+1}}{{1-\beta^{(n)}_{\lfloor sn \rfloor+1}}}-1\; ; \; t\geq0\right\}_{n\geq1}
    \]
    C-tightness of $\left\{\sum_{s_n<t_n}\mathbb{E}\left[\frac{\beta^{(n)}_{\lfloor sn \rfloor +1}}{1-\beta^{(n)}_{\lfloor sn \rfloor +1}}\Bigg|\mathcal{F}^n_{s_n}\right]-1\; ; \; t\geq0\right\}_{n\geq1}$ then follows from this, together with C-tightness of 
    \[
    \left\{\sum_{s_n<t_n}\frac{\beta^{(n)}_{\lfloor sn \rfloor+1}}{{1-\beta^{(n)}_{\lfloor sn \rfloor+1}}}-\mathbb{E}\left[\frac{\beta^{(n)}_{\lfloor sn \rfloor+1}}{{1-\beta^{(n)}_{\lfloor sn \rfloor+1}}}\Bigg|\mathcal{F}^n_{s_n}\right]\; ; \; t\geq0\right\}_{n\geq1},
    \]
    which was shown in Remark 4.2 to be a consequence of Lemma 4.4.
    
    Now, again by (1.38), we have 
    \[
    \sum_{s_n<t_n}\frac{\beta^{(n)}_{\lfloor sn \rfloor+1}}{{1-\beta^{(n)}_{\lfloor sn \rfloor+1}}}-1=\sum_{i=1}^{\lfloor nt \rfloor}-\frac{1}{D_n}\ln\left(\frac{1-\beta_i}{\beta_i}\right) + \frac{1}{2D_n^2}\left(\ln\left(\frac{1-\beta_i}{\beta_i}\right)\right)^2+o\left(\frac{1}{D_n^2}\right).
    \]
    From the above expression we see the convergence 
    \[
    \left\{\sum_{s_n<t_n}\frac{\beta^{(n)}_{\lfloor sn \rfloor+1}}{{1-\beta^{(n)}_{\lfloor sn \rfloor+1}}}-1\; ; \; t\geq 0 \right\} \Rightarrow \left\{-W_t\; ; \; t\geq0 \right\}.
    \]
    This follows from assumption (1.13), that $D_n^2\sim n^{2H}L(n)$ for some $\frac{1}{2}<H<1$ (see 3.2), and the fact that for all $i$, $\beta_i \in (v,1-v)$ for some $v\in (0,\frac{1}{2})$.
\end{proof}

\begin{corollary}
    The sequence $\{Z^n_t(\phi):t\geq0 \}$ is C-tight.
\end{corollary}
\begin{proof}
    Immediate from Lemmas 4.4 and 4.5.
\end{proof}

\begin{proposition}
    The sequence $\{X^n\}_{n\geq 1}$ is C-tight in $D_{\mathcal{M}(\mathbb{\dot{R}}^d)}[0,\infty)$.
\end{proposition}
\begin{proof}
    Immediate from Lemmas 4.1, 4.2, 4.4, 4.5, and Corollary 4.1.
\end{proof}

We now need to identify any limiting point as lying in the original space $D_{\mathcal{M}(\mathbb{{R}}^d)}[0,K]$ with the relevant continuity properties.

\begin{corollary}
    The sequence $\{X^n\}_{n\geq 1}$ is tight in $D_{\mathcal{M}(\mathbb{{R}}^d)}[0,K]$ and any limit point is continuous on $[0,K)$.
\end{corollary}
\begin{proof}
    Let $X$ be a limit point of the sequence $\{X^n\}_{n\geq 1}$ in $D_{\mathcal{M}(\mathbb{\dot{R}}^d)}[0,\infty)$, and let $\{\widetilde{X}^n\}_{n\geq1}$ be the restriction of the sequence $\{X^n\}_{n\geq1}$ to the space $\mathbb{R}^d$. We also set each $\widetilde{X}^n_K=0$. Since each $X^n$ is concentrated on $\mathbb{R}^d$, $\{\widetilde{X}^n\}_{n\geq1}$ is the sequence of measure-valued processes defined by the discrete Brownian snakes in (1.31), or equivalently by definition (1.23). By Skorohod's representation theorem, we may switch to a probability space where the convergences $X^n \to X$ in $D_{\mathcal{M}(\mathbb{\dot{R}}^d)}[0,\infty)$ (Proposition 4.1), and $\widetilde{X}^n_t(\phi) \to \int_0^{T^0_1}\phi(\mathbb{\widehat{W}}_s)L^t(ds)$ for $\phi \in \mathcal{B}(\mathbb{R}^d)$ (Theorem 1.3), hold almost surely.
    
    We know that $\widetilde{X}^n(1_{\mathbb{R}^d}){=} X^n(1_{\mathbb{\dot{R}}^d})$ almost surely on $[0,K)$. This follows as these two branching processes have identical starting mass and offspring distributions on this time interval. Using Theorem 1.3, for any $0\leq t < K$,
    \[
    X_t(1_{\mathbb{\dot{R}}^d}) {=} \int_0^{T^0_1}L^t(ds).
    \]
    Now let $\widetilde{\phi}_k \in \mathcal{C}^2_b(\mathbb{{R}}^d)$ be an increasing sequence of compactly supported functions vanishing at infinity such that $\widetilde{\phi}_k \to 1_{\mathbb{R}^d}$. Extend each $\widetilde{\phi}_k$ to $\phi_k$ defined on $\mathbb{\dot{R}}^d$ by setting $\phi_k(\infty)=0$. Note then $\phi_k\in \mathcal{C}^2_b(\mathbb{\dot{R}}^d)$ and $\phi_k \to 1_{\mathbb{R}^d}$ also. 
    
    Since each $X^n_t$ is concentrated on $\mathbb{R}^d$, by Theorem~1.3 and Proposition 4.1,
    \[
    \begin{aligned}
    X_t(\phi_k)&=\lim_{n\to \infty}X^n_t(\phi_k) \\
    &=\lim_{n\to \infty}\widetilde{X}^n_t(\widetilde{\phi}_k) \\
    &=\int_0^{T^0_1}\widetilde{\phi}_k(\mathbb{\widehat{W}}_s)L^t(ds)
    \end{aligned}
    \]
    By monotone convergence,
    \begin{equation}
    X_t(1_{\mathbb{R}^d}){=}\int_0^{T^0_1}L^t(ds).
    \end{equation}
    Hence $X_t(1_{\mathbb{\dot{R}}^d})=X_t(1_{\mathbb{R}^d})$ and $X_t(1_{\infty})=0$ almost surely. We may extend this last equality to hold almost surely on a countable dense set in $[0,K]$.
    
    Now, the map $t\mapsto X_t(1_{\mathbb{\dot{R}}^d})$ is continuous, since $1_{\mathbb{\dot{R}}^d}\in\mathcal{C}^2_b(\mathbb{\dot{R}}^d)$. Moreover, by (4.11), 
    \[
    X_t(1_{\mathbb{R}^d})=L^t([0,T^0_1])=L^t_{T^0_1}
    \]
    and hence is the weak limit of a sequence of branching processes in the random environments $(\beta^{(n)}_i)_{i\geq1}$. This can be seen from the definition of $L$ and by the same argument as in Proposition 2.2. Hence, by the result in Proposition 2.2, $t\mapsto L^t_{T^0_1}$ is continuous and so $t\mapsto X_t(1_{\mathbb{R}^d})$ is continuous. This yields $X_t(1_{\infty})=0$ a.s. on $[0,K]$.

    Continuity of $X$ on $[0,K)$ follows by C-tightness of the sequence $\{X^n\}_{n\geq1}$ in $D_{\mathcal{M}(\mathbb{\dot{R}}^d)}[0,\infty)$.
\end{proof}

All that remains is the derivation of the martingale problem. In view of Lemmas 4.1 and 4.4, we would like to show any limiting point of the sequence $\{Z^n(\phi)\}_{n\geq1}$ is a martingale and calculate it's quadratic variation.

\begin{lemma}
For any $t>0$, the sequence $\{Z^n_t(\phi)\}_{n\geq1}$ is uniformly bounded in $L_2$.
\end{lemma}
\begin{proof}
    Using the conditional independence of the random variables $N^{n,\alpha}$ across distinct $\alpha$, we see
    \[
    \begin{aligned}
    \mathbb{E}[Z^n_t(\phi)^2]&\leq \frac{||\phi^2||_{\infty}}{n}\sum_{s_n<t_n}\mathbb{E}\left[X^n_{s_n}(1)\left(\frac{\beta^{(n)}_{\lfloor sn \rfloor +1}}{(1-\beta^{(n)}_{\lfloor sn \rfloor +1})^2}\right)\right] \\
    &\leq \frac{||\phi^2|_{\infty}v_{max}}{n}\sum_{s_n<t_n}\mathbb{E}[X^n_{s_n}(1)]
    \end{aligned}
    \]
    where $\frac{\beta^{(n)}_{i}}{(1-\beta^{(n)}_{i})^2}\leq v_{max}$ for all $i$ independently of $n$ by the uniform boundedness of the $\beta_i$, and hence the $\beta^{(n)}_i$. Applying Lemma 3.4 we see
    \[
    \mathbb{E}[Z^n_t(\phi)^2]\leq \frac{||\phi^2|_{\infty}v_{max}}{n}\sum_{s_n<t_n}1+C(s),
    \]
   for some suitable constant $C(s)$. Since $t$ is fixed we may let $\sup_{s\leq t}C(s)\leq C< \infty$ for some universal constant $C$. The result follows.
\end{proof}

\begin{corollary}
    Let $Z(\phi)$ be a limiting point of the sequence $\{Z^n(\phi)\}_{n\geq 1}$. Then $Z(\phi)$ is a martingale. 
\end{corollary}
\begin{proof}
    Immediate from Lemma 4.6.
\end{proof}

\begin{lemma}
    The sequence $\{\langle Z^n(\phi)\rangle\}_{n\geq1}$ is C-tight and any limiting point has the representation $\langle Z(\phi)\rangle _t=2\int_0^tX_s(\phi^2)ds$.
\end{lemma}
\begin{proof}
    We will show the convergence directly. For any particle $\alpha$ born at time $s_n-a_n$ and hence reproducing at time $s_n$, let $m^n_{s_n}=\mathbb{E}[N^{n,\alpha}|\beta^{(n)}]=\frac{\beta^{(n)}_{\lfloor sn \rfloor}}{1-\beta^{(n)}_{\lfloor sn \rfloor}}$. Then we have
    \[
    \begin{aligned}
    \langle Z^n(\phi)\rangle_t&=\frac{1}{n^2}\sum_{s_n<t_n}\mathbb{E}\left[\left(\sum_{\alpha \sim_n s_n}\phi(B^{n,\alpha}_{s_n+a_n})(N^{n,\alpha}-m^{n}_{s_n+a_n})\right)^2\Bigg|\mathcal{F}^n_{s_n}\right] \\
    &=\frac{1}{n^2}\sum_{s_n<t_n}\mathbb{E}\left[\sum_{\alpha \sim_n s_n}\phi(B^{n,\alpha}_{s_n+a_n})^2(N^{n,\alpha}-m^n_{s_n+a_n})^2|\mathcal{F}^n_{s_n}\right] \\
    &=\frac{1}{n^2}\sum_{s_n<t_n}\sum_{\alpha\sim_n s_n}P_{{1}/{n}}\phi^2(B^{n,\alpha}_{s_n})\mathbb{E}\left[\frac{\beta^{(n)}_{\lfloor sn \rfloor +1}}{(1-\beta^{(n)}_{\lfloor sn \rfloor +1})^2}\Bigg|\mathcal{F}^n_{s_n}\right] \\
    &=\frac{1}{n}\sum_{s_n<t_n}X^n_{s_n}(P_{{1}/{n}}(\phi^2))\mathbb{E}\left[\frac{\beta^{(n)}_{\lfloor sn \rfloor +1}}{(1-\beta^{(n)}_{\lfloor sn \rfloor +1})^2}\Bigg|\mathcal{F}^n_{s_n}\right]
    \end{aligned}
    \]
    where $(P_t)_{t\geq0}$ is the heat semigroup. Again we used the conditional independence of the $\{N^{n,\alpha}\}_{\alpha\sim_n s_n}$ given the environment in the second equality, and the independence of the spatial motion and the environment in the third.
    
    Now, again using (4.7), it follows that
    \[
    \{\sum_{s_n<t_n}X^n_{s_n}(P_{{1}/{n}}(\phi^2))\; ; \; t\geq0\}_{n\geq1}\Rightarrow \{X_t(\phi^2)\; ; \; t\geq0 \}.
    \]
    By the uniform-in-$i$ boundedness of the $\beta^{(n)}_i$ it also follows that
    \[
    \frac{1}{n}\sum_{s_n<t_n}\mathbb{E}\left[\frac{\beta^{(n)}_{\lfloor sn \rfloor +1}}{(1-\beta^{(n)}_{\lfloor sn \rfloor +1})^2}\Bigg|\mathcal{F}^n_{s_n}\right] \to 2t
    \]
    almost surely. Hence, using the continuity of $t\mapsto X_t(\phi^2)$ and the fact that the process 
    \[
    t\mapsto \frac{1}{n}\sum_{s_n<t_n}\mathbb{E}\left[\frac{\beta^{(n)}_{\lfloor sn \rfloor +1}}{(1-\beta^{(n)}_{\lfloor sn \rfloor +1})^2}\Bigg|\mathcal{F}^n_{s_n}\right]
    \]
    is increasing, we see 
    \[
    \left\{\frac{1}{n}\sum_{s_n<t_n}X^n_{s_n}(P_{\frac{1}{n}}(\phi^2))\mathbb{E}\left[\frac{\beta^{(n)}_{\lfloor sn \rfloor +1}}{(1-\beta^{(n)}_{\lfloor sn \rfloor +1})^2}|\mathcal{F}^n_{s_n}\right]\; ; \; t\geq0 \right\} _{n\geq1}\Rightarrow \left\{2\int_0^tX_s(\phi^2)ds\; ; \; t\geq0 \right\},
    \]
    and we are done.
\end{proof}

\textbf{Proof of Theorem 1.2:}
\begin{proof}
The tightness of the sequence $\{X^n\}_{n\geq1}$ in $D_{\mathcal{M}(\mathbb{{R}}^d)}[0,K]$ and continuity property was shown in Corollary 4.2. Moreover, by Lemmas 4.1, 4.2, 4.4, 4.5, 4.6, 4.8 and Corollary 4.3 we see that any limiting point $X$ satisfies the martingale problem
\[
X_t-X_0-\frac{1}{2}\int_0^tX_s(\Delta\phi)ds-\gamma_t(\phi)
\]
is a continuous martingale on $[0,K)$ with quadratic variation
\[
2\int_0^tX_s(\phi^2)ds,
\]
where $\gamma_t(\phi)$ is a limiting point of the sequence $\{A^n_t(\phi)\; ; \; t\geq0\}_{n\geq1}.$ The theorem is proved.
\end{proof}

We conclude this section with some comments comparing the form of the martingale problem we have just proved to that stated in [8]. 

As mentioned in Section 1.5, we recall that in the case considered in [8], the environment $(\xi^n_i(x))_{i\geq 1}$ is a sequence of i.i.d random fields across $\mathbb{R}^d$, with mean $\frac{\nu}{\sqrt{n}}$ and covariance function $g(x,y)$ independent of $n$. The offspring distribution for a particle located at $x$ at time $t=\frac{i}{n}$ is geometric with parameter $\frac{1}{2}-\frac{\xi^n_i(x)}{4\sqrt{n}}$. Then, in the notation above, for a particle $\alpha$ born at the branching time $s_n$, $N^{n,\alpha}$ is independent of $\mathcal{F}^n_{s_n}$ and 
\[
\mathbb{E}[N^{n,\alpha}-1|\mathcal{F}^n_{s_n+a_n}]=1+\frac{\xi^n_i(x)}{\sqrt{n}}+\frac{(\xi^n_i(x))^2}{2n}+\text{ higher order terms.}
\]
Hence the discrete integrator term in the process $A^n_t(\phi)$ defined in (4.6) is
\[
\mathbb{E}[N^{n,\alpha}-1|\mathcal{F}^n_{s_n}]=\mathbb{E}[N^{n,\alpha}-1]=\frac{\nu}{n}+\frac{\bar{g}}{2n}+\text{higher order terms}. 
\]
where $\bar{g}=g(x,x)$. From this observation (and ignoring higher order terms that vanish in the limit), it follows that for the case considered in [8],
\[
\begin{aligned}
A^n_t(\phi)&=\frac{1}{n}\sum_{s_n<t_n}\sum_{\alpha \sim_n s_n}\phi(B^{n,\alpha}_{s_n+a_n})(\mathbb{E}[N^{n,\alpha}|\mathcal{F}^n_{s_n}]-1) \\
&=\frac{1}{n}\sum_{s_n<t_n}\sum_{\alpha \sim_n s_n}\phi(B^{n,\alpha}_{s_n+a_n})(\frac{\nu}{n}+\frac{\bar{g}}{2n}) \\
&=\frac{1}{n}\sum_{s_n<t_n}X^n_{s_n+a_n-}\left((\nu+\frac{1}{2}\bar{g})\phi\right)
\end{aligned}
\]
and so clearly 
\[
\{A^n_t(\phi)\; ; \; t\geq0\}\Rightarrow \left\{\int_0^tX_s\left((\nu+\frac{1}{2}\bar{g})\phi\right)ds\; ; \; t\geq0\right\}
\]
and the martingale problem (1.34) mentioned without proof by the authors in [8] is recovered.

In our model, it will not always be the case that the sequence $\{A^n_t(\phi)\; ; \; t\geq0\}_{n\geq1}$ converges to an integral process, as we do not require that the limit of the integrator processes, $W$, is a process against which stochastic integration is well-defined. However, as an example, in the case that $W$ is a semimartingale satisfying some natural conditions (see [4] for the details of the conditions), we shall have
\[
\{A^n_t(\phi)\; ; \; t\geq0\}\Rightarrow \left\{\int_0^tX_s(\phi)dW_s\; ; \; t\geq0\right\} 
\]
in $D[0,\infty)$, and thus the martingale problem becomes
\[
X_t-X_0-\frac{1}{2}\int_0^tX_s(\Delta\phi)ds-\int_0^tX_s(\phi)dW_s
\]
is a continuous martingale with quadratic variation
\[
2\int_0^tX_s(\phi^2)ds,
\]
a very natural generalisation of the martingale problem for the model considered in [8].

\section{An explicit example}
In this final short section we give an explicit example of a sequence $(\beta_i)_{i \in \mathbb{Z}}$ satisfying conditions (1.12), (1.13) and (1.24). The example is taken from [14].

Take a stationary sequence of Gaussian random variables $\{X_i\}_{i\geq 1}$ with $\mathbb{E}[X_i]=0, \mathbb{E}[X_i^2]=1$, and $\mathbb{E}[X_iX_{i+k}]\sim k^{2H-2}L(k)$ for some slowly varying function $L$ and $\frac{1}{2}<H<1$. Note in this case that $\sum_{k=1}^{\infty}\mathbb{E}[X_iX_{i+k}]=\infty$.

Let $G$ be a bounded function of Hermite rank 1 (see [14] Chapter 3 for the definition of Hermite rank) such that $\mathbb{E}[G(X_i)]=0, \mathbb{E}[G(X_i)^2]<\infty$. Take the sequence $(\beta_i)_{i\geq1}$ such that 
\[
\ln\left(\frac{1-\beta_i}{\beta_i}\right)=G(X_i).
\]
\begin{remark}
    From the asymptotic form of the correlation kernel of the $X_i$, we see that the sequence $(\beta_i)_{i \geq 1}$ displays long range memory.
\end{remark}
Since $G$ is bounded, we must have that the $\beta_i$ take values in $(v,1-v)$ for some $v\in (0,\frac{1}{2}).$ By the definition of $G$, the conditions (1.12) are immediately satisfied. Then, by Theorem 4.1 and Lemma 5.1 in [14], 
\[
\left\{\frac{1}{D_n}\sum_{i=1}^{\lfloor nt \rfloor}\ln\left(\frac{1-\beta_i}{\beta_i}\right)\; ; \; t\geq0\right\}\Rightarrow \{B^H_t\; ; \; t\geq0\}
\]
in $D[0,\infty)$ where $D_n^2\sim n^{2H}L(n)$, and $B^H$ is fractional Brownian motion with Hurst parameter $H$. Thus condition (1.13) is satisfied. Since fractional Brownian motion is continuous, the assumption in (1.24) can also be seen to be satisfied.
\begin{remark}
In [14] the weak convergence is shown to hold in $D[0,1]$. The extension to $D[0,\infty)$ that we state here is easily seen to be true.
\end{remark}
In this example, the process $W$ mentioned throughout this paper is therefore fractional Brownian motion with Hurst parameter $H>\frac{1}{2}$. We also assumed that the potential $W$ is conservative (see (1.20) for the definition of a conservative potential), so that the $W$-associated process defined in (1.19) exists as a random element of $C[0,\infty).$ Theorem 1.1 in [9] concerning the growth rate of certain continuous Gaussian processes implies that $\limsup_{t\to \infty}B^H_t=\infty$, from which it is immediate that $A(y):=\int_0^ye^{B^H_z}dz$ has $A(\infty)=\infty, A(-\infty)=-\infty$ and as such is bijective. Therefore $B^H$ is a conservative potential.

Hence all the conditions are satisfied by this example. We call the $B^H$-associated process Brownian motion in a fractional Brownian potential, and the resulting superprocess $X$ super-Brownian motion in a fractional Brownian potential. 

\textbf{\Large{Acknowledgements}}\newline

This paper was written as part of a summer research project under the supervision of Professor Alison Etheridge at the Department of Statistics, University of Oxford, UK. I am very grateful for all of the invaluable help and guidance she has provided during the preparation of this work. I am also indebted to Dr. Jo\~{a}o De Oliveira Madeira for his incredibly helpful comments on a first draft, and to Ruairi Garrett and Julio Ernesto Nava Trejo for many fruitful discussions in this area.\newline

\textbf{\Large{Glossary}}\newline

Here we list frequently used notation. In the second column of the table we give a brief description, and in the third we refer to the section where the notation is defined.

\begin{longtable}{c p{10cm} c}
\toprule
\textbf{Notation} & \textbf{Meaning} & \textbf{Section} \\
\midrule
\endfirsthead

\toprule
\textbf{Notation} & \textbf{Meaning} & \textbf{Section} \\
\midrule
\endhead
$(\beta_i)_{i\in \mathbb{Z}}$ & The environment for our model. & 1.2 \\
$D_n$ & The scaling constant for the discrete potential $\sum_{i=1}^{n}\ln\left(\frac{1-\beta_i}{\beta_i}\right)$. & 1.2 \\
$W$ & The weak limit of $\left\{\sum_{i=1}^{\lfloor nx \rfloor}\ln\left(\frac{1-\beta_i}{\beta_i}\right)\; ; \; x\in \mathbb{R}\right\}_{n\geq1}$. & 1.2 \\
$(\beta^{(n)}_i)_{i\in \mathbb{Z}}$ & The rescaled environments & 1.2 \\
$\widetilde{S}^n$ & The Donsker-rescaled random walk in the random environment $(\beta^{(n)}_i)_{i\in \mathbb{Z}}$. & 1.2 \\
$\widetilde{M}^n$ & The rescaled branching process under the first $n$ excursions of $\widetilde{S}^n$. & 1.2 \\
$Z$ & A potential (a real-valued process). & 1.2 \\
$A_Z$ & The scale function for a potential $Z$. & 1.2 \\
$dM_Z$ & The speed measure for a potential $Z$. & 1.2 \\
$Y$ & The $W$-associated process. & 1.2 \\
$X^n$ & a BBM/BBMRE in $\mathbb{R}^d$. The branching mechanism will always be specified. & 1.3 \\
$N^n_i$ & A random variable with distribution Geom($1-\beta^{(n)}_i$), representing the offspring distribution of $X^n$ defined in (1.23) at time $t=\frac{i}{n}$. & 1.3 \\
$\mathcal{M}(\mathbb{R}^d)$ & The non-negative, finite measures on $\mathbb{R}^d$ with the weak topology. & 1.3 \\
$\mathcal{M}_R(\mathbb{R})$ & The non-negative Radon measures on $\mathbb{R}$, with the vague topology. & 3.1 \\
$S^n$ & The Donsker-rescaled random walk in the random environment $(\beta^{(n)}_i)_{i\geq1}$, reflected at 0 and $K$. & 1.4 \\
$\mathcal{W}$ & The set of all stopped paths & 1.4 \\
$\mathbf{W}^n_{\cdot}$ & The $C_{\mathbb{R}^d}[0,\infty)$ path-valued process obtained by concatenating or erasing independent Brownian motions. & 1.4 \\
$\mathbb{W}^n_{\cdot}=(\mathbf{W}^n_{\cdot},S^n_{\cdot})$ & The Brownian snake process with path process $\mathbf{W}^n_{\cdot}$ and lifetime process $S^n_{\cdot}$. & 1.4 \\
$\mathbb{\widehat{W}}^n_{\cdot}$ & The terminal point of the snake process $\mathbb{W}^n_{\cdot}$. & 1.4 \\
$L^{n,s}_t$ & The local time of $S^n$ at level $s$ by time $t$. & 1.4 \\
$T^{n,a}_t$ & The inverse local time of $S^n$ at level $a$ by time $t$. & 1.4 \\
$W^n$ & The discrete potential function for $\widetilde{S}$. & 2 \\
$Y^n$ & The $W^n$-associated process defined by speed and scale. & 2 \\
$m^n_i$ & The mean of the offspring distribution at time $t=\frac{i}{n}$ of the BBMRE $X^n$/branching process $\widetilde{M}^n$. & 2 \\
$\eta$ & The Feller diffusion & 2 \\
$(P_t)_{t\geq0}$ & The heat semigroup & 3.1 \\
$\alpha=(\alpha_0,\alpha_1,...,\alpha_N)$ & A labeling of the first $N$ descendants of the particle $\alpha_0$ alive at time $t=0$. We take $\alpha$ to refer to the particle $\alpha_N$, which will branch at the next branching time. & 4 \\
$\alpha \sim_nt$ & A particle $\alpha$ that will branch at the next branching time after time $t$. & 4 \\
$N^{n,\alpha}$ with $\alpha \sim_nt$ & The number of offspring of the particle $\alpha$ at the next branching time after time $t$. & 4 \\
$\mathcal{F}^n_t$ & The right-continuous filtration generated by the BBMRE $X^n$, which includes the environments up to time $t$. & 4 \\
$M^{n,\alpha,s_n}_{s_n+a_n}(\phi)$ & The martingale noise due to the spatial motion of the particle $\alpha\sim_n s_n$ over its lifetime $[s_n,s_n+a_n)$. & 4 \\
$Z^{n}(\phi)$ & The martingale noise process due to the branching conditioned on the environment in the process $X^n(\phi)$, for a test function $\phi$. & 4 \\
$N^n(\phi)$ & The martingale noise process due to the environment in the process $X^n(\phi)$, for a test function $\phi$. & 4 \\
$M^n(\phi)$ & The martingale noise process due to the spatial motion of the process $X^n(\phi)$, for a test function $\phi$. & 4 \\
$A^n(\phi)$ & The aggregate drift process due to the branching conditioned on the environment and the fluctuations of the environment. & 4 \\
\bottomrule
\end{longtable}

\textbf{\Large{References}}\newline

[1] A. Etheridge (2000).\textit{ An Introduction to Superprocesses}. University Lecture Series, vol. 20, American Mathematical Society.

[2] W. Hong, H. Yang and K.Zhou (2015). Scaling limit of local time of Sinai's random walk. \textit{Front. Math. China.}\textbf{ 10,} 1313-1324.

[3] T. Kurtz (1978). Diffusion approximations for branching processes. \textit{Adv. Prob. Relat. Topics, vol. 5,} 269-292.

[4] T. Kurtz and P. Protter (1991). Weak limit theorems for stochastic integrals and stochastic differential equations. \textit{Ann. Prob. vol. 19}\textbf{ 3}, 1035-1070

[5] J. Lamperti (1962). Semi-stable stochastic processes. \textit{Trans. Amer. Math. Soc. vol. 104}, 62-78

[6] J.-F. Le Gall (1999).\textit{ Spatial branching processes, random snakes and partial differential equations.} Lectures in Mathematics ETH Zurich, Birkhauser Verlag, Basel.

[7] L. Mytnik (1996). Superprocesses in random environments.\textit{ Ann. Prob. vol. 24,}\textbf{ 4}, 1953–1978.

[8] L. Mytnik, J. Xiong and O. Zeitouni (2011). Snake representation of a superprocess in a random environment.\textit{ ALEA, Lat. Am. J. Probab. Math. Stat. vol. 8}, 335-378.

[9] S. Orey (1972). Growth rate of certain Gaussian processes.\textit{ Proceedings of the Sixth Berkeley Symposium on Mathematical Statistics and Probability, Volume II: Probability Theory}, 443-451, University of California Press.

[10]  E. Perkins (2002). Dawson–Watanabe superprocesses and measure-valued diffusions.\textit{ Lectures on Probability Theory and Statistics, SaintFlour 1999,} Lecture notes in Mathematics,\textbf{ 1781}, 132–329, Springer, Berlin.

[11] D. Revuz and M. Yor (1980).\textit{ Continuous Martingales and Brownian Motion}. Springer.

[12] L. Rogers (1984). Brownian local times and branching processes.\textit{ Séminaire de Probabilités, vol. 18}, 42-55.

[13] P. Seignourel (2000). Discrete schemes for processes in random media.\textit{ Probab. Theory Relat. Fields. vol. 118} 293-322.

[14] M. Taqqu (1975). Weak convergence to Fractional Brownian Motion and to the Rosenblatt Process.\textit{ Zeitschrift für Wahrscheinlichkeitstheorie und verwandte Gebiete (later renamed Probability Theory and Related Fields), vol. 31}, 287-302.
\end{document}